\documentclass{amsart}
\usepackage[utf8]{inputenc}

\usepackage{geometry}
\usepackage{enumerate}
\usepackage{amsmath}
\usepackage{cite}
\usepackage{ulem}
\usepackage{amsfonts}
\usepackage{mathrsfs}
\usepackage{galois}
\usepackage[toc,page]{appendix}

\newtheorem{Def}{Definition}[section]
\newtheorem{Thm}[Def]{Theorem}

\newtheorem{Lem}[Def]{Lemma}

\newtheorem{Cor}[Def]{Corollary}
\newtheorem{Rem}[Def]{Remark}

\numberwithin{equation}{section}

\newcommand{\AAa}{\mathbb{A}}
\newcommand{\BB}{\mathbb{B}}

\newcommand{\NN}{\mathbb{N}}

\newcommand{\RR}{\mathbb{R}}
\newcommand{\SSp}{\mathbb{S}}

\newcommand{\ZZ}{\mathbb{Z}}

\newcommand{\cB}{\mathcal{B}}

\newcommand{\cF}{\mathcal{F}}

\newcommand{\cI}{\mathcal{I}}

\newcommand{\cK}{\mathcal{K}}

\newcommand{\cO}{\mathcal{O}}

\newcommand{\cU}{\mathcal{U}}
\newcommand{\cV}{\mathcal{V}}
\newcommand{\cW}{\mathcal{W}}

\newcommand{\cZ}{\mathcal{Z}}

\newcommand{\scA}{\mathscr{A}}

\newcommand{\scE}{\mathscr{E}}

\newcommand{\scH}{\mathscr{H}}

\newcommand{\scM}{\mathscr{M}}
\newcommand{\scN}{\mathscr{N}}

\newcommand{\scR}{\mathscr{R}}

\newcommand{\scU}{\mathscr{U}}

\newcommand{\scX}{\mathscr{X}}

\newcommand{\bfC}{\textbf{C}}

\newcommand{\bfF}{\textbf{F}}
\newcommand{\bfG}{\textbf{G}}

\newcommand{\bfI}{\textbf{I}}

\newcommand{\bfL}{\textbf{L}}
\newcommand{\bfM}{\textbf{M}}

\newcommand{\wsto}{\overset{\ast}{\rightharpoonup}}

\author{Zhihan Wang} 
\address{Department of Mathematics, Princeton University, Princeton, NJ 08544}
\email{zhihanw@math.princeton.edu}

\title{Min-max Minimal Hypersurfaces with Obstacle}

\begin{document}

  \begin{abstract}
   We study min-max theory for area functional among hypersurfaces constrained in a smooth manifold with boundary. A Schoen-Simon-type regularity result is proved for integral varifolds which satisfies a variational inequality and restricts to a stable minimal hypersurface in the interior. Based on this, we show that for any admissible family of sweepouts $\Pi$ in a compact manifold with boundary, there always exists a closed $C^{1,1}$ hypersurface with codimension$\geq 7$ singular set in the interior and having mean curvature pointing outward along boundary realizing the width $\bfL(\Pi)$.
  \end{abstract}

  \maketitle   

    \section{Introduction} 
    
     Min-max theory for minimal hypersurfaces in a closed manifold is established by Almgren-Pitts in \cite{Almgren65_Varifolds, Pitts81} and by Schoen-Simon \cite{SchoenSimon81} for regularity in higher dimensions. In recent years, the abundance of closed minimal hypersurfaces has been well studied by \cite{MarquesNeves17_Infinite, IrieMarquesNeves18, MarquesNevesSong19, Song18_YauConj, LiYy19_Exis}.     
     
     We consider in this paper the following min-max problem in a manifold with boundary.\\

   \textbf{Question:} Can we find a \textit{minimal} hypersurface constrained in a compact smooth domain $M$ in an oriented Riemannian manifold $(\tilde{M}, g)$? \\
     
     Clearly, the answer is no in general if one only focus on closed surfaces with vanishing mean curvature. Allowing an extension of subjects we would like to find, two distinct aspects of this problem are frequently studied.\\
     
\noindent   \textbf{$\bullet$ Free boundary problem} 
     A minimal hypersurface with free boundary in $M$ is a hypersurface with boundary in $\partial M$, perpendicular to $\partial M$ and having vanishing mean curvature. The study of free boundary minimal surfaces has a long history, we refer the readers to \cite{ZhouLi16_FreeBdyMinmax} for a beautiful survey of the development. \\

\noindent   \textbf{$\bullet$ Obstacle problem}
     Let $(\tilde{M}, g) = (\RR^{n+1}, |dx|^2)$, $E_0\subset \RR^{n+1}$ be a bounded Caccioppoli set; Consider a minimizer $\hat{E}$ of the perimeter of Caccioppoli sets among \[
       \{E\subset \RR^{n+1}: E\Delta E_0 \subset M \}    \]
     The regularity of boundary of $\hat{E}$ is well studied in \cite{Miranda71, BarozziMassari82, Tamanini82, Lin85}, and a non-parametric version is studied in \cite{CaffarKinderle80, BrezisKinderl74, LewyStampacchia71, GiaquintaPepe71} under weaker assumptions on $M$. We emphasis here that in general, the optimal regularity for $\partial \hat{E}$ is $C^{1,1}$ instead of $C^2$ where $\partial \hat{E}$ touch $\partial M$, and the mean curvature of $\partial \hat{E}$ does not necessarily vanish. We refer the readers to \cite{Giusti72} for a brief introduction and generalizations when $M$ is not necessarily smooth. For later use in this paper, see subsection \ref{Subsection_Obst}.
     
     The min-max analogue of obstacle problem for area functional and other parametric elliptic integral is also studied recent year under further convexity assumption on $M$, and hence the stationarity of solutions valid by maximum principle. See \cite{Song18_YauConj, BernsteinWang20, DeLellisRamic18, Montezuma18}.\\
     
     The first goal of this paper is to prove the following,
     \begin{Thm} \label{Thm_Intro_Exist Minmax Obst}
     Let $m\geq 1$, $n\geq 2$, $\cZ_n(M)$ be the space of integral cycles supported in $M$, endowed with $\bfF$-metric (see section \ref{Subsection_Notations}), $\Phi_0: I^m\to (\cZ_n(M), \bfF)$ be a continuous map. Let $\Pi$ be the family of sequences of $\bfF$-continuous maps $S = \{\Phi_i: I^m\to \cZ_n(M)\}_{i\geq 1}$ which $C^0$-tends to $\Phi_0|_{\partial I^m}$ restricted on $\partial I^m$. Let \[
      \bfL(\Pi):= \inf_{S = \{\Phi_i\}_{i\geq 1}\in \Pi} \limsup_{i\to \infty}\sup_{x\in I^m} \bfM(\Phi_i(x)); \ \ \ \ \bfL_{\Phi_0, \partial I^m} := \sup_{x\in \partial I^m} \bfM(\Phi_0(x))     \]
     If $\bfL(\Pi)>\bfL_{\Phi_0, \partial I^m}$, then there exists a \textbf{constrained embedded minimal hypersurface} in $M$ with the same mass as $\bfL(\Pi)$ (see definition \ref{Def_Constr emb_hypersurface}). In particular, there's an immersed $C^{1,1}$ hypersurface $\Sigma\hookrightarrow M$ such that 
     \begin{enumerate}
     \item[•] $\scH^n(\Sigma) = \bfL(\Pi)$
     \item[•] $\Sigma$ is minimal, locally stable and embedded with multiplicity and codim $\geq 7$ singularity in the interior of $M$;
     \item[•] $\Sigma$ is multiple $C^{1,1}$ graphs near $\partial M$, and the mean curvature vector of $\Sigma$ on $\Sigma\cap \partial M$ points outward, i.e. \[
       \vec{H}_{\Sigma}\cdot \nu_M \leq 0    \]
       for inward normal field $\nu_M$ of $\partial M$
     \end{enumerate}
     \end{Thm}
     For the precise statement, see theorem \ref{Thm_Minmax_Exist constr min} and the notations in section \ref{Section_Pre}.
     
     Our proof is based on the idea of min-max construction in \cite{Pitts81} and the regularity result in \cite{SchoenSimon81}. Two major difficulties are the lack of Schoen-Simon-type regularity for stable varifolds near $\partial M$, and the invalidity of unique continuation result for minimal hypersurface with obstacle. To fix this, we propose a general structure theorem for a certain class of varifolds $V$ near $\partial M$, saying that locally $V$ is the sum of finitely many $C^1$-graphs over $\partial M$, see theorem \ref{Thm_Strc thm_Main}. The improved regularity of min-max solution then follows from a standard replacement argument together with a uniform $C^{1,1}$ estimate for graphs satisfying a variational inequality. Based on this, the compactness is also proved for a family of constrained embedded stable minimal hypersurfaces, see corollary \ref{Cor_Str thm_Cptness of c-stable min}.
     
     A natural question is when the hypersurface we construct above has vanishing mean curvature. By strong maximum principle \cite{White09}, this is the case when $M$ is weakly mean convex. Without mean convexity, in section \ref{Section_Loc rigid}, we prove the following rigidity theorem,
     \begin{Thm} \label{Thm_Intro_Loc Rigid}
      Let $\Sigma\subset N^{n+1}$ be a two-sided, non-degenerate smooth closed minimal hypersurface; $\epsilon \in (0,1)$. Then there's $s_0 = s_0(\Sigma, N, \epsilon)>0$ such that if $\delta\in(0, s_0)$, then any constrained embedded minimal hypersurface in the $\delta$-neighborhood $Clos(B_{\delta}(\Sigma))$ of $\Sigma$ with the mass between $\scH^n(\Sigma)$ and $(2-\epsilon)\scH^n(\Sigma)$  is $\Sigma$ itself.
     \end{Thm}
     In view of the local min-max property by \cite{White94}, we can see from theorem \ref{Thm_Intro_Exist Minmax Obst} and theorem \ref{Thm_Intro_Loc Rigid} that a non-degenerate two-sided smooth minimal hypersurface could be reconstruct from min-max in a sufficiently small neighborhood. In an upcoming work, we generalize the approach above to construct minimal hypersurfaces lying nearby certain singular minimal hypersurface under any sufficiently small perturbation of metric.
     
     \subsection*{Organization of the paper.} In section \ref{Section_Pre}, we introduce some notations and preliminaries for constrained stationary varifold, an object we shall deal with throughout this note. We also include a brief survey on obstacle problem in the minimizing setting in \ref{Subsection_Obst} and two $C^{1,1}$-regularity theorems we shall use later, the proof of which are in the appendix; section \ref{Section_Str thm} is devoted to the proof of main structure theorem for constrained stationary varifolds that are stable in the interior. In section \ref{Section_Minmax}, we discuss the min-max constructions and prove a precise version of theorem \ref{Thm_Intro_Exist Minmax Obst}. Section \ref{Section_Loc rigid} discusses theorem \ref{Thm_Intro_Loc Rigid}.
     
    \subsection*{Acknowledgement}
     I am grateful to my advisor Fernando Cod\'a Marques for his constant support and guidance. I'm thankful to Chao Li and Xin Zhou for suggesting this problem and inspiring discussions and to Fanghua Lin for helpful explainations to his thesis \cite{Lin85} on obstacle problem. I would also thank Yangyang Li, Zhenhua Liu and Lu Wang for their interest in this work.

    \section{Preliminaries} \label{Section_Pre}
     \subsection{Notations} \label{Subsection_Notations}
     Here are some notations we adopt throughout the article. In an $L$-dimensional Euclidean space $\RR^L$, 
     \begin{align*}
      \BB^L_r(p)\ &\ \text{the open ball of radius } r \text{ centered at }p \\
      \SSp^{L-1}_r(p)\ &\ \text{the sphere of radius } r \text{ centered at }p \\
      \AAa_{s, r}^L(p)\ &\ \text{the open annuli }\BB_r(p)\setminus \overline{\BB_s(p)} \text{ centered at }p \\
      \eta_{p, r}\ &\ \text{the map between }\RR^L, \text{ maps }x \text{ to }(x-p)/r \\
      \bfG(L, k)\ &\ \text{the Grassmannian of } k\text{-dimensional unoriented linear subspace in }\RR^L \\
      \scH^k\ &\ k\text{-dimensional Hausdorff measure} \\
      \omega_k\ &\ \text{the volume of }k \text{-dimensional unit ball} \\
      \theta^k(x, \mu)\ &\ \text{the k-th density (if exists) of a Radon measure }\mu \text{ at }x, \text{ i.e. }\lim_{r\to 0}\frac{1}{\omega_k r^k}\mu(\BB_r^L(x))
     \end{align*}

     Let $(M, g)$ be a compact oriented $n+1$-dimensional smooth manifold with boundary $\partial M \neq \emptyset$, $n\geq 2$; For sake of simplicity, we can always extend $(M, g)$ to a closed $n+1$-manifold $(\tilde{M}, \tilde{g})$ isometrically embedded in $\RR^L$. We always work under the intrinsic topology of $\tilde{M}$. Denote the intrinsic Riemannian metric by $\langle\cdot , \cdot\rangle$ and the Levi-Civita connection by $\nabla$. Also write 
     \begin{align*}
      Int(A)\ &\ \text{the interior of a subset }A\subset \tilde{M} \\
      Clos(A)\ &\ \text{the closure of a subset }A\subset \tilde{M} \\
      dist_M\ &\ \text{the intrinsic distant function on }M \\
      dist_{\partial M}\ &\ \text{the intrinsic distant function on }\partial M \\
      B_r(p)\ &\ \text{the open geodesic ball in } \tilde{M} \text{ of radius } r \text{ centered at }p \\
      B_r(E)\ &\ \text{the r neighborhood of }E\subset \tilde{M},  \bigcup_{p\in E} B_r(p) \\
      S_r(p)\ &\ \text{the geodesic shpere in } \tilde{M} \text{ of radius } r \text{ centered at }p \\
      A_{s, r}(p)\ &\ \text{the open geodesic annuli in } \tilde{M}, B_r(p)\setminus \overline{B_s(p)}  \\
      \cB_r(p)\ &\ \text{the open geodesic ball in } \partial M \text{ of radius} r \text{ centered at }p \\
      \nu_{\partial M}\ &\ \text{the inward unit normal field of }\partial M \text{ with respect to }M,\\
        &\ \text{ extended to a neighborhood of }\partial M \text{ by }\nabla_{\nu_{\partial M}}\nu_{\partial M} = 0 \\
      \vec{H}_{\partial M}\ &:=\ -div_{\partial M} (\nu_{\partial M}) \cdot \nu_{\partial M} = - H_{\partial M}\cdot\nu_{\partial M},\ \text{ the mean curvature vector of }\partial M \\
      A_{\partial M}\ &:=\ \nabla \nu_{\partial M},\ \text{ the 2nd fundamental form of }\partial M\\
      T_p^+M\ &:=\ \{v\in T_pM : v\cdot \nu_{\partial M} \geq 0\},\ \text{ the inward pointed vectors in }T_pM \\
      \exp ^M\ &\ \text{the exponential map of }\tilde{M} \\
      inj(\partial M)\ &\ \text{the intrinsic injectivity radius of }\partial M \\
      \scX_c(U)\ &\ \text{the space of compactly supported vector fields in a relative open subset }U\subset M \\
      \scX_c^+(U)\ &:=\ \{X\in \scX_c(U): X_p\in T_p^+M,\ \forall p\in \partial M \} \\
      \scX_c^{tan}(U)\ &:=\ \{X\in \scX_c(U): X_p\in T_p \partial M,\ \forall p\in \partial M \} 
     \end{align*} 

     Let $r_M$ be the relative injective radius of $\partial M$ in $\tilde{M}$,
     \begin{align} 
      r_M:= \inf{\big\{}inj(\partial M), \sup\{r: (x, s)\mapsto \exp^M_x(s\cdot \nu_{\partial M}(x)) \text{ is a diffeomorphism on }\partial M\times (-r, r) \} {\big\}}  
      \label{Pre_Notation_rel inj radi}
     \end{align}
     We shall work under Fermi coordinates in $B_{r_M}(\partial M)$, \[
      \Phi: \partial M\times (-r_M, r_M) \to B_{r_M}(\partial M),\ \ (x, s)\mapsto \exp^M_x(s\cdot \nu_{\partial M}(x))     \] 
     and shall simply write a point $p=\Phi(x, s)\in B_{r_M}(\partial M)$ as $(x, s)$. Under this notation, for $x\in \partial M$ and $r\in (0, r_M)$ write
     \begin{align*}
      C_r(x)\ &\ \text{the open cylinder in }\tilde{M}, \{(y, s): dist_{\partial M}(y, x) <r, |s|<r\}\\
      C_r^+(x)\ &\ \text{the open half cylinder in }M,  \{(y, s): dist_{\partial M}(y, x) <r, 0<s<r\}\\
      \partial^+ C_r(x)\ &\ \text{the top boundary of }C_r(x),  \{(y, r): dist_{\partial M}(y, x) <r\}\\
      \partial^s C_r(x)\ &\ \text{the side boundary of }C_r(x),  \{(y, s): dist_{\partial M}(y, x) =r, |s|<r\}
     \end{align*}
     For $u\in C^1(\cU, (-r_M, r_M))$ for some domain $\cU\subset \partial M$, define \[
      graph_{\partial M}(u):= \{(x, u(x)): x\in \cU\}    \]
     and orient it by choosing the normal field to $graph_{\partial M}(u)$ having positive inner product with $\nu_M$.

     We also recall some basic notions of currents and varifolds and refer readers to the standard references \cite{Federer69}, \cite{Simon83_GMT} for details. The following are the spaces we shall work with,
     \begin{align*}
      \bfI_k(M)\ &\ \text{the space of }k\text{-dimensional integral currents in }\RR^L \text{ with support in }M \\
      \cZ_k(M)\ &:=\ \{T\in \bfI_k(M): \partial T = 0\} \\
      \cV_k(M)\ &\ \text{the closure of }\cZ_k(M)\text{ under varifold weak topology} \\
      \cI\cV_k(M)\ &\ \text{the space of integral k-varifolds in }\tilde{M} \text{ supported in }M
     \end{align*}

     For an immersed k-dimensional submanifold $\Sigma$ in $M$ with finite volume, let $|\Sigma|$ be the associated integral k-varifold in $\cI\cV_k(M)$; If further $\Sigma$ is oriented, let $[\Sigma]$ be the associated integral k-current. For $T\in \bfI_k(M)$, let $|T|$ and $\|T\|$ be the associated integral varifold and Radon measure on $M$.
     
     For a relatively open subset $U\subset M$, let $\cF_U$, $\bfM_U$ be the flat metric and mass norm on $\bfI_k(U)$, and $\bfF$ be the metric compatible with weak convergence of varifolds in $M$ on $\cI\cV_k(M)$. If $U = M$, we shall omit the subscript. The $\bfF$-metric between integral currents is defined by $\bfF(T, S)= \cF(T, S)+ \bfF(|T|, |S|)$, $S, T\in \bfI(M)$. 
     
     For a proper $C^1$ map $f$ between $\RR^L$, $f_{\sharp}$ be the push forward of varifold or current;
     
     For $E\subset M$ and $V\in \cV_k(M)$, $V\llcorner E$ be the restriction of $V$ onto $E$.
     
     For $X\in \scX^+(M)$ and $V\in \cV_k(M)$, let \[
      \delta V(X) := \frac{d}{dt}\bigg|_{t=0} \bfM(e^{tX}\ _{\sharp}V) = \int div^SX(x)\ dV(x, S)    \]
     be the first variation of $V$, where $e^{tX}$ be the one-parameter family of diffeomorphism induced by $X$; $div^S X:=\sum\langle \nabla_{e_i}X, e_i\rangle$, $\{e_1, ..., e_k\}$ be an orthonormal basis of $S$.

    \subsection{Constrained stationary varifolds} \label{Subsection_Constr stny vfd}
     Throughout this subsection, let $M$ be as in subsection \ref{Subsection_Notations}, $1\leq k\leq n$, $U\subset M$ be a bounded relatively open subset. We note that the following definitions also generalize directly to properly embedded $M\subset \RR^L$, not necessarily compact.
     \begin{Def}
      Call $V\in \cV_k(M)$ stationary in $U$ if \[
       \delta V(X) = 0\ \ \ \forall X\in \scX_c(U)     \]
      Call $V\in \cV_k(M)$ \textbf{constrained stationary} along $\partial M$ in $U$ if \[
       \delta V(X) \geq 0\ \ \ \forall X\in \scX_c^+(U)     \]
      Call $V$ \textbf{stationary with free boundary} along $\partial M$ in $U$ if \[
       \delta V(X) = 0\ \ \ \forall X\in \scX_c^{tan}(U)   \]
     \end{Def}
     We shall omit the description \textsl{along $\partial M$} in both case if there's no ambiguity.
     
     The same condition of a constrained stationary varifold is studied in \cite{White09}, where it is called \textsl{varifold minimizes area to first order}. The notion of stationary varifold with free boundary was introduced in \cite{GuangLiZhou18Crelle} for the study of free boundary minimal surfaces. Since $\scX_c^{tan}(U)\subset \scX_c^+(U)$ and $X\in \scX_c^{tan}(U)$ iff $-X\in \scX_c^{tan}(U)$ by definition, every constrained stationary varifold in $U$ is stationary with free boundary.
     
     We list some basic results on constrained stationary varifolds, which will be used in the next few sections. 
     \begin{Lem}[Compactness] \label{Lem_Constr stny_Cptness}
      Let $(M_j, g_j)\subset \RR^L$ be a family of manifold with boundary converging to $(M, g)$ in $C^2_{loc}$, $\scU \subset \RR^L$ be an open subset such that $U = \scU\cap M$; Let $V_j\in \cV_k(M_j)$ be constrained stationary in $\scU\cap M_j$ and $V_j\wsto V \in \cV_k(\RR^L)$. Then, $V\in \cV_k(M)$ is constrained stationary in $U$. 
     \end{Lem}
     \begin{proof}
      By the continuity of $\delta V (X)$ in $V$ w.r.t. varifold topology and in $X$ w.r.t. to $C^1$ topology. 
     \end{proof}
     
     \begin{Lem}[Boundary Monotonicity] \label{Lem_Constr stny_Bdy Mon}
      Let $\cK\subset U\cap \partial M$ be a compact subset. Then there exists $R_0=R_0(M, U, \cK)>0$, $\Lambda = \Lambda(M)>0$ such that if $V\in\cV_k(M)$ is a stationary varifold with free boundary along $\partial M$ in $U$, $p\in \cK$, then the function \[
       r\mapsto \frac{e^{\Lambda r}}{r^k}\|V\|(\BB^L_r(p))   \] 
      is non-decreasing in $r\in (0, R_0)$. Moreover, when $(U, M, g)= (\RR^{n+1}_+, \RR^{n+1}_+, |dx|^2)$, if $\|V\|(B_r(p))/r^k$ is constant in $r\in (0,  =\infty)$, then $V$ is a cone, i.e. \[
       (\eta_{p, r})_{\sharp} V = V \ \ \ \forall r\in (0, +\infty)    \]
     \end{Lem}
     This is a corollary of \cite[theorem 3.4]{GuangLiZhou18Crelle}, where $N = \partial M$ and an explicit monotonicity formula is proved.\\
     
     \begin{Cor} \label{Cor_Constr stny_Bdy density lbd}
      Suppose $V\in \cI\cV_k(M)$ is a stationary integral varifold with free boundary along $\partial M$ in $U$; $x\in spt(V)\cap U\cap \partial M$, then \[
       \theta^k(x, \|V\|) \geq 2^{-k}   \]
     \end{Cor}
     \begin{proof}
      Since $V$ is integer multiplicity $k$-rectifiable, there exist (under Fermi coordinates) $p_j=(x_j, t_j)\to x$ such that $\theta^k(p_j, \|V\|)\geq 1$. Let $r_j:=|x-x_j|$. Suppose WLOG that either $t_j \equiv 0$ or $t_j>0$ for all $j\geq 1$.
      For every $0<r<<1$, take $j>>1$ such that $r-r_j>2t_j$. In the first case, by lemma \ref{Lem_Constr stny_Bdy Mon} \[
       \frac{1}{\omega_k r^k}\|V\|(\BB^L_r(x))\geq \limsup_{j\to \infty} \frac{1}{\omega_k (r-r_j)^k}\|V\|(\BB^L_{r-r_j}(x_j)) \geq 1   \] 
      While in the second case,
      \begin{align*}
       \frac{1}{\omega_k r^k}\|V\|(\BB^L_r(x))
       \geq &\ \limsup_{j\to \infty} \frac{1}{\omega_k (r-r_j)^k}\|V\|(\BB^L_{r-r_j}(x_j)) \\
       \geq &\ \limsup_{j\to \infty} \frac{e^{\Lambda 2t_j}}{\omega_k (2t_j)^k}\|V\|(\BB^L_{2t_j}(x_j)) \\
       \geq &\ \limsup_{j\to \infty} \frac{e^{\Lambda 2t_j}}{2^k\cdot \omega_k t_j^k}\|V\|(\BB^L_{t_j}(p_j)) 
       \geq 2^{-k}     
      \end{align*}
      where the second inequality follows from lemma \ref{Lem_Constr stny_Bdy Mon}, and the last inequality follows from interior monotonicity formula for stationary varifolds. Let $r\to 0$, the corollary is proved.  
     \end{proof}
     
     \begin{Lem}[Strong maximum principle] \label{Lem_Constr stny_Maxm Principle}
      Suppose $\cW:= U\cap \partial M$ is connected and mean convex, i.e. $\vec{H}_{\partial M}\cdot \nu_{\partial M}|_x \geq 0$; $V\in\cI\cV_n(M)$ be a constrained stationary integral varifold along $\partial M$ in $U$. Suppose $spt(V)\cap \cW \neq \emptyset$, then
      \begin{enumerate}
       \item[(1)] $\cW\subset spt(V)$ and $\vec{H}_{\partial M} = 0$ on $\cW$. 
       \item[(2)] Moreover, $V\llcorner U = V_b + V_I$ where $V_b= m|\cW|$ for some $m\in \NN$ and $spt(V_I)\cap W = \emptyset$.
      \end{enumerate}             
     \end{Lem}
     \begin{proof}
      (1) is proved in \cite[theorem 4]{White09}, so is (2) under the additional assumption that $V$ is stationary. The only place stationarity is used is that $x\in spt(V)\cap \cW$ implies $\theta^n(x, \|V\|) \geq 1$, which also holds for constrained stationary integral varifolds in mean convex domain by combining (1) and the proof of corollary \ref{Cor_Constr stny_Bdy density lbd}.
     \end{proof}
     
     The following lemma guarantees that constrained stationary varifolds have locally bounded first variation. Recall $r_M$ is defined in (\ref{Pre_Notation_rel inj radi}). 
     \begin{Lem} \label{Lem_Constr stny_Loc bd 1st var}
      Let $V\in \cV_k(M)$ be constrained stationary in $U$ with $\|V\|(U)\leq \Lambda$; $U'\subset \subset U$. Then $\exists\ C = C(M, U, U')>0$ s.t. \[
       |\delta V(X)| \leq C\|V\|(U)\cdot \|X\|_{C^0}\ \ \ \forall X\in \scX_c(U')  \]
     \end{Lem}
     \begin{proof}
      We shall work under Fermi coordinte in $B_{r_M}(\partial M)$, where $\nu_{\partial M}$ is extended by $\nabla_{\nu_{\partial M}}(\nu_{\partial M}) = 0$. Let $\eta\in C^1_c(B_{r_M}(\partial M)\cap U), [0,1])$ and equal to $1$ near $\partial M\cap U'$. For every $X\in \scX_c(U')$ with $\|X\|_{C^0}\leq 1$, let $X_0 := X - \eta\cdot \langle X, \nu_{\partial M}\rangle\nu_{\partial M}$. Notice that $\pm X_0 \in \scX_c^+(U)$ and $1\pm \langle X, \nu_{\partial M}\rangle \geq 0$, hence
      \begin{align*}
       \delta V(X) &= \int div^S X_0 + div^S(\eta\cdot \langle X, \nu_{\partial M}\rangle\nu_{\partial M})\ dV(x, S) \\
       & = \int div^S(\eta\cdot (\langle X, \nu_{\partial M}\rangle -1 )\nu_{\partial M})\ dV(x, S) + \int div^S(\eta\cdot \nu_{\partial M})\ dV(x, S) \\
       & \leq \int div^S(\eta\cdot \nu_{\partial M})\ dV(x, S) \leq C(M, U, U')\|V\|(U).
      \end{align*}
      where the first inequality follows from that $\eta\cdot (1- \langle X, \nu_{\partial M}\rangle )\nu_{\partial M} \in \scX_c^+(U)$. Similarly, 
      \begin{align*}
       \delta V(X) &= \int div^S(\eta\cdot (\langle X, \nu_{\partial M}\rangle +1 )\nu_{\partial M})\ dV(x, S) - \int div^S(\eta\cdot \nu_{\partial M})\ dV(x, S) \\
       &\geq \int div^S(\eta\cdot \nu_{\partial M})\ dV(x, S) \geq -C(M, U, U')\|V\|(U).
      \end{align*}
      where $C(M, U, U'):=\|\eta\cdot \nu_{\partial M}\|_{C^1}$.
     \end{proof}
     
     \begin{Cor} \label{Cor_Constr stny_Tan at Bdy}
      Let $V\in \cI\cV_n(M)$ be a constrained stationary integral varifold in $U$. Then $\forall x\in spt(V)\cap U\cap \partial M$, the tangent varifold of $V$ at $x$ exists and equals to $m|T_x\partial M|$ for some $m\in \NN$.
     \end{Cor}
     \begin{proof}
      By lemma \ref{Lem_Constr stny_Cptness}, \ref{Lem_Constr stny_Bdy Mon}, \ref{Lem_Constr stny_Loc bd 1st var} and Allard compactness theorem  \cite[Chapter 8, 5.8]{Simon83_GMT}, up to a subsequence of $r_j\to 0$, $(\eta_{x, r_j})_{\sharp}V \wsto C \in \cI\cV(T_x^+(M))$. Moreover, by the rigidity part of lemma \ref{Lem_Constr stny_Bdy Mon}, $C$ is a constrained stationary cone along $T_x(\partial M)$ centered at $0$. Since $0\in spt(C)$, by lemma \ref{Lem_Constr stny_Maxm Principle}, $C = m|\partial M| + C_I$ for some stationary varifold $C_I$ supported in $T_x^+M \setminus T_x(\partial M)$ and $m\in \NN$. But since $C_I$ is also a cone, $C_I = 0$ and $C=m|\partial M|$ and hence is unique since $m = \theta^n(x, \|V\|)$ is independent of the converging sequences. 
     \end{proof}

    \subsection{Constrained embedded hypersurfaces} \label{Subsection_Constr emb}
     A major problem in the regularity theory of constrained stationary varifolds is that one should not expect embeddedness near $\partial M$. Consider for example, $M = \RR^{n+1}\setminus \BB^{n+1}_1(0)$, $V = |\SSp^n_1(0)|+|\{x_{n+1}=1\}|\in \cI\cV_n(M)$ is constrained stationary in $M$ but is supported on an immersed submanifold with self-touching points. We introduce in this subsection the following notion of constrained embedded hypersurfaces in order to describe the optimal regularity for constrained stationary varifolds. Suppose $U\subset M$ be a relatively open subset. Recall $Int(U) = U\setminus \partial M$.
     
     \begin{Def} \label{Def_Constr emb_hypersurface}
      Call $V\in \cI\cV_n(U)$ a \textbf{constrained embedded hypersurface} with optimal regularity if $\forall p\in U$, there's a neighborhood $U_p$ of $p$ in $U$, $m\in \NN\cup\{0\}$ and $C^{1,1}$ embedded hypersurfaces $\Sigma_j\subset U_p$ without boundary, $1\leq j\leq m$ such that 
      \begin{enumerate}
       \item[(1)] $V\llcorner U_p = \sum_{j=1}^m |\Sigma_j|$,
       \item[(2)] $\partial M \cap U_p \cap Clos(\Sigma_j)\setminus \Sigma_j = \emptyset$ and $\scH^{n-2}(U_p\cap Clos(\Sigma_j)\setminus \Sigma_j) = 0$,
       \item[(3)] If $q\in U_p\cap Clos(\Sigma_i)\cap Clos(\Sigma_j)$, then either $q\in \partial M$ or $Clos(\Sigma_i)\cap Int(U_p)$ coincide with $ Clos(\Sigma_j)\cap U_p$ on the connected component containing $q$.
      \end{enumerate}
      Let $Sing(V):= \{p\in spt(V)\cap Int(U): p \text{ is not a regular point of hypersurfaces above in }U_p\}$.
     \end{Def}
     Here a hypersurface $\Sigma\subset M$ is called $C^{1, \alpha}$ if locally it's the graph of some $C^{1,\alpha}$ function over some $n$-plane. 
     
     \begin{Rem} \label{Rem_Constr emb_Bd mean curv}
     (1) When a constrained embedded $V$ is constrained stationary along $\partial M$, it's clear that $spt(V)\cap Int(U)$ is a minimal hypersurface in $Int(U)$; Also for every $x\in \partial M\cap U$, suppose for some neighborhood $U_x$ of $x$ that $V\llcorner U_x = \sum_j|graph_{\partial M}(u_j)|$, since $u_j$ are $C^{1,1}$ and by lemma \ref{Lem_Constr stny_Maxm Principle}, $\{u_j = 0\}\subset \{y\in \partial M : H_{\partial M}\geq 0\}$. 
     Hence $\Delta_{\partial M}u_j=0$ a.e. $x\in \{u_j = 0\}$, $|graph_{\partial M}(u_j)|$ is constrained stationary in $U$ along $\partial M$ and has $L^{\infty}$-mean curvature $H_{\partial M}\cdot \chi_{\{u_j=0\}}$. 
     
     (2) Also note that a constrained embedded hypersurface $V$ could be viewed as the associated varifold of some immersed hypersurface. More precisely, consider $\Sigma\subset M\times \NN$ defined by \[
      \Sigma := \{(p,j) :  p\in spt(V)\setminus Sing(V), 1\leq j \leq \theta^n(x, \|V\|)\}     \]
      and endow topology on $\Sigma$ by specifying neighborhoods of points in $\Sigma$:
      \begin{enumerate}
       \item[•] If $(p, j)\in int(M)\times \NN\cap \Sigma$, $\cO\subset \Sigma$ is a neighborhood of $(p, j)$ iff $\exists U\ni p$ an open subset of $M$ such that $U\times\{j\}\subset \cO$;
       \item[•] If $(p, j)\in \partial M\times \NN\cap \Sigma$, let $U_p$ be a neighborhood of $p$ in $M$ such that $V\llcorner U_p = \sum_{i=1}^{m_p}|graph_{\partial M}(u_i)|$, where $0\leq u_1\leq u_2\leq ... \leq u_{m_p}$ are $C^{1,1}$ functions over $U_p\cap \partial M$, if $u_i(x) = u_j(x)$, then either they both vanish or $u_i = u_j$ on the connected component of $\{u_i >0\}$ containing $x$. By definition of $\Sigma$, $m_p\geq j$; for each $x\in U_p\cap \partial M$, define $l_j(x):= \inf\{l\in \NN: u_l(x) = u_j(x)\}$.
       
       Now, $\cO\subset \Sigma$ is a neighborhood of $(p, j)$ iff $\exists \cU\subset \partial M\cap U$ an open neighborhood of $p$ in $\partial M$ such that \[
        \{(x, u_j(x), j-l_j(x)+1 ): x\in \cU\} \subset \cO     \]
       where recall that $(x, u_j(x))$ represent a point in $spt(V)$ under Fermi coordinate near $\partial M$.
      \end{enumerate}
      We leave it to readers to verify that such $\Sigma \to M$ by projection onto first invariant is a $C^{1,1}$ immersed submanifold. 
     \end{Rem}
     
     \begin{Def} \label{Def_Constr emb_converg}
      Let $\{V_j\}, V$ be constrained embedded hypersurfaces with optimal regularity in $U$. Call $V_j$ \textbf{$C^1$-converges} to $V$ if $\forall p\in spt(V)\setminus Sing(V)$, there's some neighborhood $U_p\cap Sing(V)=\emptyset$ of $p$ such that
      \begin{enumerate}
       \item[(1)] If $p\in Int(U)$, then $U_p\cap \partial M = \emptyset$, $V\llcorner U_p = m|\Sigma|$ for some $m\in \NN$, $\Sigma\subset U_p$ properly embedded hypersurface and $V_j\llcorner U_p = \sum_{i=1}^m |graph_{\Sigma}(v_j^{(i)})|$ for some $v_j^{(i)}$ with $|v_j^{(i)}|_{C^1(\Sigma)}\to 0$ as $j\to \infty$.
       \item[(2)] If $p\in \partial M$, then $V\llcorner U_p = \sum_{i=1}^m |graph_{\partial M}(u^{(i)})|$, $V_j\llcorner U_p = \sum_{i=1}^m |graph_{\partial M}(u_j^{(i)})|$ for some $m\in \NN$, $u^{(i)}, u_j^{(i)}\in C^{1,1}(\partial M\cap U_p)$ and $|u^{(i)}-u_j^{(i)}|_{C^1}\to 0$. 
      \end{enumerate}
     \end{Def}

    \subsection{Obstacle problem} \label{Subsection_Obst}
     Here we only focus on minimal hypersurfaces with smooth obstacle.  Problems in two different settings are well studied, mainly for hypersurfaces in the Euclidean spaces.
     
     \textbf{Non-parametric obstacle problem} Let $\Omega\subset \RR^n$ be a smooth bounded convex domain. $\psi \in C^2(\bar{\Omega})$, $\varphi\in C^2(\partial \Omega)$. Consider the problem 
     \begin{enumerate}
      \item[(N)] Minimize \[
       \scA(u):=\int_{\Omega} \sqrt{1+|\nabla u|^2}\ dx   \]
       among all Lipschitz functions $u$ such that $u=\varphi$ on $\partial \Omega$ and $u\geq \psi$ in $\Omega$.
     \end{enumerate}
     
     The existence and uniqueness of solutions to problem (N) was studied by \cite{LewyStampacchia71} and \cite{GiaquintaPepe71} under the weaker assumptions on $\psi$, $\varphi$ and $\Omega$. See also \cite{Giusti72}.
     
     The regularity of minimizer $\bar{u}$ in problem (N) was also studied in \cite{LewyStampacchia71} and \cite{GiaquintaPepe71}. Assuming $\psi\in H^{2,p}$, they are able to prove $\bar{u}\in H^{2, p}$, $\forall n< p <+\infty$. Using an approximating argument, \cite{BrezisKinderl74} shown that $\bar{u}$ is $C^{1,1}$ provided $\psi\in C^2(\bar{\Omega})$. 
     
     A linearization of problem (N) is studied in \cite{CaffarKinderle80}, where they considered the minimizer of enegry functional $\scE(u):= \int_{\Omega} |\nabla u|^2\ dx$ instead of area functional $\scA$, and obtained $C^{1, \alpha}$-regularity for minimizers by a Harnack typed inequality only assuming $\psi\in C^{1,\alpha}(\bar{\Omega})$, $0\leq \alpha\leq 1$. Their approach was generalized in \cite{Lin85} to deal with the non-parametric obstacle problem for area functional. 
     
     The regularity result we shall use in the rest of this paper is
     \begin{Thm} [Interior Regularity] \label{Thm_Survey Obst_Reg for non-para}
      Let $M\subset \RR^L$ be as in subsection \ref{Subsection_Notations}, $\alpha \in (0,1)$. Then there exists $r_0=r_0(M) \in (0, r_M /2)$ such that, if $r<r_0$, $y\in \partial M$, $u\in C^1(\cB_r(y), [0,r))\cap C_{loc}^{1,\alpha}(\cB_r(y)\setminus \{y\})$ with $|\nabla u|\leq 1$ and $|graph_{\partial M}(u)|$ constrained stationary in $C_r(y)$ along $\partial M$, then $u\in C^{1,1}_{loc}(\cB_r(y))$ and for every compact subset $\cK\subset \cB_r(y)$, \[
       |\nabla u(x) - \nabla u(x')|\leq C(M, \cK, r_0)|x-x'|\ \ \ \text{ for every }x, x'\in \cK   \] 
     \end{Thm}
     
     For later applications, we include the case with isolated singularity. The proof essentially follows \cite{Lin85}. For some technical reason, we have to assume a priori $C^{1,\alpha}$ regularity instead of only $C^1$ regularity. But the point is we don't need a priori a uniform $C^{1,\alpha}$ bound. For sake of completeness, we include in the Appendix the proof of theorem \ref{Thm_Survey Obst_Reg for non-para} and the following theorem \ref{Thm_Survey Obst_Reg for para}. \\

     \textbf{Parametric obstacle problem}
      Let $\Omega\subset \RR^{n+1}$ be a bounded domain, $E_0\subset\RR^{n+1}$ be a set of finite perimeter. Consider the following obstacle problem for minimizing boundary
     \begin{enumerate}
      \item[(P)] Minimize $\bfM(\partial [E])$ among all Cacciappoli sets $E\subset \RR^{n+1}$ with $spt([E]-[E_0])\subset \bar{\Omega}$.
     \end{enumerate}
      
     Regularity of solution $\bar{E}$ to Problem (P) was first studied by \cite{Miranda71}, in which he shown $\partial_*\bar{E}$ is of class $C^1$ near $\partial \Omega$ provided $\partial \Omega$ is $C^1$ ( Call a hypersurface \textsl{of class $C^{k, \alpha}$} if locally it's the graph of some $C^{k,\alpha}$ function ). Combined with the result in non-parametric problem, $\partial_*\bar{E}$ is $C^{1,\alpha}$ if $\partial \Omega$ is $C^{1,\alpha}$, $0\leq \alpha\leq 1$. Different approaches were exploited in \cite{BarozziMassari82} and \cite{Tamanini82} where $C^{1,\alpha}$ regularity ($\alpha\in (0, 1)$) of $\partial_*\bar{E}$ is proved assuming $H^{2,p}$ or $C^{1,\alpha}$ regularity of $\partial \Omega$ correspondingly. We asserts that \cite{Tamanini82} proves a general regularity result on reduced boundary with good excess estimate (see also \cite{Almgren75, Bombieri82} ), hence directly generalize to minimizing boundary with obstacle in a manifold. The following regularity for minimizing currents with obstacle in a compact manifold with boundary will be used in section \ref{Section_Minmax}.

     \begin{Thm} \label{Thm_Survey Obst_Reg for para}
      Let $M\subset \RR^L$ be as in subsection \ref{Subsection_Notations}, $U \subset M$ be a relative open subset with non-empty intersection with $\partial M$. Suppose $T\in \cZ_n(U)$ minimizes mass in $\{S\in \cZ_n(U) : spt(T-S)\subset\subset U\}$. Then $\forall y\in U\cap \partial M$, there's a neighborhood $U_y$ of $y$ and $C^{1,1}$ functions $u_1\geq u_2\geq ... \geq u_q\geq 0$, $|\nabla u_j|\leq 1$ such that 
      \begin{enumerate}
       \item[(1)] $\pm T\llcorner U_y = \sum_{j=1}^q [graph_{\partial M}(u_j)]\llcorner U_y$. 
       \item[(2)] For every $x\in U_y\cap \partial M$, if $u_i(x)=u_j(x)$, then either $u_i(x)=u_j(x) = 0$, or $u_i \equiv u_j$ on the connected component of $\{u_i>0\}$ containing $x$.
      \end{enumerate}
     \end{Thm}

   \section{A structure theorem near the boundary} \label{Section_Str thm}
    The goal of this section is to prove the following structure theorem of a class of constrained stationary varifold near $\partial M$. Throughout the section, $U\subset M$ be a relatively open subset with nonempty intersection with $\partial M$, $\cK\subset \partial M\cap U$ be a compact subset. 
    Keep working under Fermi coordinates near $\partial M$. Use $x, y, z$ to denote points on $\partial M$ and $p = (x, t), q=(y, s)$ to denote points in $Int(M)$. We shall not distinguish between $x$ and $(x, 0)$.
    
    \begin{Def}
     Call $V\in \cV_n(Int(U))$ a \textbf{stable minimal hypersurface with multiplicity}, if there exists a countable family of disjoint properly embedded stable minimal hypersurfaces $\{\Sigma_j\}_{j\geq 1}$ with optimal regularity in $Int(U)$ and $m_j\in \NN$ such that
     \begin{enumerate}
      \item[•] $\bigcup_{j\geq 1} \Sigma_j$ has no accumulation point in $Int(U)$, i.e. $Clos(\bigcup_{j\geq 1} \Sigma_j)\cap Int(U) = \bigcup_{j\geq 1} \Sigma_j$.
      \item[•] $V = \sum_j m_j|\Sigma_j|$
     \end{enumerate}
     Call points in $\bigcup_{j\geq 1}Reg(\Sigma)$ \textbf{regular points} of $V$ in $Int(U)$.\\
     
     Call $V\in \cI\cV(U)$ \textbf{constrained stable varifold} if $V$ is constrained stationary in $U$ and a stable minimal hypersurface with multiplicity in $Int(U)$. Call it \textbf{constrained embedded stable minimal hypersurface} if in addition it's a constrained embedded hypersurface.
    \end{Def}
    
    \begin{Thm} \label{Thm_Strc thm_Main}
     Suppose $V\in \cI\cV_n(M)$ be a constrained stable varifold in $U$ along $\partial M$ with $\|V\|(U)\leq \Lambda$. 
     
     Then $\exists\ r_0=r_0(M, U, \cK, \Lambda)\in (0, dist(\cK, U^c)/4)$, $m_0 = m_0(M, U, \cK, \Lambda)< +\infty$ such that, for every $y\in \cK$, \[
      V\llcorner C_{r_0}(y) = V_0 + V_+   \]  
     where $V_+$ is a stable minimal hypersurfaces with multiplicity in $C_{r_0}(y)$ and $spt(V_+)\cap B_{r_0/2}(\partial M) = \emptyset$; $V_0 = \sum_{j=1}^m|graph_{\partial M}(u_j)|$ for some $0\leq m\leq m_0$ and $C^1$ functions $u_j: \cB_{r_0}(y) \to [0, 3r_0/4)$ satisfying 
     \begin{enumerate}
      \item[(1)] $spt(V_0)\cap spt(V_+) = \emptyset$;
      \item[(2)] $u_1\geq u_2\geq ...\geq u_m$, $|\nabla u_j|\leq 1/2$;
      \item[(3)] If $y'\in \cB_{r_0}$ such that $u_i(y') = u_j(y')$, then either $u_i(y')=0$ or $u_i\equiv u_j$ on the connected component of $\{u_i>0\}$ containing $y'$.
     \end{enumerate} 
     If further, $u_j \in C^{1,1}(\cB_{r_0}(y))$, then $|graph_{\partial M}(u_j)|$ is a constrained stationary varifold in $C_{r_0}(y)$.
    \end{Thm}
    An interesting analytic question is whether one can directly conclude that $u_j\in C^{1,1}$.
    
    In view of theorem \ref{Thm_Survey Obst_Reg for non-para}, a direct corollary of theorem \ref{Thm_Strc thm_Main} is the compactness of constrained embedded stable minimal hypersurfaces. The proof is left to readers.
    \begin{Cor} \label{Cor_Str thm_Cptness of c-stable min}
     Let $\{V_j\}_{j\geq 1}$ be a sequence of constrained embedded stable minimal hypersurfaces in $U$ along $\partial M$, with $\|V_j\|(U)\leq \Lambda$. Then up to a subsequence, there exists a constrained embedded stable minimal hypersurface $V_{\infty}\in \cI\cV_n(U)$ such that $V_j$ converges to $V$ in varifold sense and in $C^1$ sense (see definition \ref{Def_Constr emb_converg}).
    \end{Cor}
    
    The rest of this section is devoted to the proof of theorem \ref{Thm_Strc thm_Main}. We start with a tilt estimate of constrained stable varifolds near $\partial M$.
    \begin{Lem} \label{Lem_Strc thm_tilt est}
     Let $V$ be in theorem \ref{Thm_Strc thm_Main}. There exists a continuous function $\rho: (0, 1/2)\to (0, dist(\cK, U^c)/2)$ depending only on $M, U, \cK, \Lambda$ such that,  for every $\epsilon \in (0,1/2)$, if $p=(y, t)\in \cK\times (0, \rho(\epsilon)) \cap spt(V)$, then $V$ is regular near $p$ and $1-|\nu_V\cdot \nu_{\partial M}|(p)^2 < \epsilon$, where $\nu_V$ denotes the unit normal of $V$.
    \end{Lem}
    \begin{proof}
     We shall argue by contradiction. If the statement is not true, then there's some $\epsilon>0$, $V_i\in \cI\cV_n(M)$ constrained stationary in $U$, stable minimal hypersurfaces in $Int(U)$ with $\|V_i\|(U)\leq \Lambda$ and $p_i=(y_i, t_i)\in spt(V)\cap Int(U)$ with $t_i\to 0$, $y_i \in \cK$ such that one of the following holds,
     \begin{enumerate}
      \item[(a)] $p_i$ is a singular point of $V\llcorner Int(U)$
      \item[(b)] $p_i$ is a regular point but $1-|\nu_V\cdot \nu_{\partial M}|(p)^2 \geq \epsilon$
     \end{enumerate}
     Suppose $y_i\to y_0 \in\cK$. By lemma \ref{Lem_Constr stny_Loc bd 1st var}, \ref{Lem_Constr stny_Bdy Mon}, \ref{Lem_Constr stny_Cptness} and Allard compactness \cite{Simon83_GMT}, up to a subsequence,\[
      (\eta_{y_i, t_i})_{\sharp}V_i \wsto P    \]
     for some $P\in \cI\cV_n(T_{y_0}^+M)$ constrained stationary along $T_{y_0}\partial M$. Thus by lemma \ref{Lem_Constr stny_Maxm Principle}, $P = k|T_{y_0}\partial M| + P_I$ for some $k\in \{0\}\cup \NN$ and stationary integral varifold $P_I$ in $T_{y_0}M$ with $spt(P_I)\subset Int(T_{y_0}^+ M)$. 

     On the other hand, since $V_i\llcorner Int(U)$ are stable minimal hypersurfaces, by Schoen-Simon compactness \cite{SchoenSimon81}, $P_I$ is a stable minimal hypersurface in $T_{y_0}M$. Since $spt(P_I)\subset Int(T_{y_0}^+ M)$, by strong maximum principle \cite{Ilmanen96}, $P_I = \sum_j |L_j|$ for some hyperplane $L_j \subset T_{y_0}M$ parallel to $T_{y_0}\partial M$. 
     
     Now that, consider (up to a subsequence) $\eta_{y_i, t_i}(p_i) \to p_0 \in Int(T_{y_0}^+ M)$. Since $p_0$ is a smooth point of $P_I$ with tangent plane parallel to $T_{y_0}\partial M$, by Schoen-Simon's $\epsilon$-regularity theorem \cite[theorem 1]{SchoenSimon81},  neither of (a) (b) is true.
    \end{proof}
    
    From now on, let $\rho$ be the function in lemma \ref{Lem_Strc thm_tilt est} with respect to some $\cK' \subset U\cap \partial M $ such that $\cK\subset Int(\cK')$.
    We shall also assume, WLOG, that $\rho < min\{1, dist(\cK, (\cK')^c)\}$. 
    The following is a direct corollary of lemma \ref{Lem_Strc thm_tilt est}, and the proof is left to readers.
    \begin{Cor} \label{Cor_Strc thm_Loc Graph}
     For every $k\in \NN$ and $p=(y, t)\in \cK \times (0, \rho(1/(10k^2))/2) \cap spt(V)$, the connected component of $spt(V)\cap C_{kt}(y)$ containing $p$ is the graph of some function $u\in C^{\infty}(\cB_{kt}(y))$ with \[
      0<|u|<2t,\ \ \ |\nabla u|< 1/(2k)     \]    
    \end{Cor}
    
    Now we proceed the proof of Theorem \ref{Thm_Strc thm_Main}. For $y\in \cK$, let $r_0 = \rho(1/4000)/2$. Suppose $V\llcorner C^+_{3r_0}(y) = \sum_{j\geq 1}m_j|\Sigma_j|$ where $\Sigma_j$ are connected, disjoint stable minimal hypersurfaces. Let $J:= \{j\geq 1: \Sigma_j\cap B_{r_0/2}(\partial M) = \emptyset\}$; $V_+ := \sum_{j\in J}m_j|\Sigma_j|$ and $V_0 = V\llcorner C_{3r_0}(y) - V_+$. By corollary \ref{Cor_Strc thm_Loc Graph}, $spt(V_0)\subset B_{3r_0/4}(\partial M)$, $spt(V_0)\cap spt(V_+) = \emptyset$ and $V_0$ is also constrained stationary in $C_{3r_0}(y)\cap M$ along $\partial M$.
    
    Define $m: \cB_{r_0}(y)\to \ZZ\cup \{+\infty\}$ by,
    \begin{align}
     m(x):= \begin{cases} 0\ &\ \text{if }V_0 = 0 \\
      \sum_{p=(x, t)\in spt(V_0)} \theta^n(p, \|V_0\|)\ &\ \text{if }V_0\neq 0
     \end{cases}
    \end{align}     
    Note that by corollary \ref{Cor_Constr stny_Tan at Bdy} and lemma \ref{Lem_Strc thm_tilt est}, $\theta^n(p, \|V\|)\in \ZZ$ for $p\in spt(V_0)\cap C_{r_0}(y)$, hence $m(x)\in \ZZ\cup \{+\infty\}$. We shall show that $m$ is constant on $\cB_{r_0}(y_0)$, bounded uniformly from above.
    \begin{Lem}
     $m(x)\leq m_0$ for some $m_0=m_0(M, U, \cK, \Lambda, r_0)$ and every $x\in \cB_{r_0}(y)$.
    \end{Lem}
    \begin{proof}
     For every $x\in \cB_{r_0}(y)$ fixed, let $\{t>0: (x, t)\in spt(V_0)\} = \{t_1>t_2>... (>0)\}$, $t_1<3r_0/4$; and let $n_i:= \theta^n((x, t_i), \|V_0\|)\in \NN$. By corollary \ref{Cor_Strc thm_Loc Graph}, for each $i\geq 1$, the connected component of $spt(V_0\llcorner C_{2t_i}(x))$ containing $(x, t_i)$, denoted by $S_i$, is a smooth minimal graph over $\cB_{2t_i}(x)$, and $\{S_i\}_{i\geq 1}$ are disjoint. Moreover \[
      \|S_i\|(\BB_{2t_i}(x))\geq e^{-\Lambda_1 t_i}\cdot \omega_n t_i ^n    \]
     for some $\Lambda_1 = \Lambda(M)\geq 0$ by interior monotonicity of area.\\
     Also, by lemma \ref{Lem_Constr stny_Bdy Mon}, suppose WLOG $\Lambda_1\geq 0$ such that $t\mapsto e^{\Lambda_1 t}\|V_0\|(\BB_t(x))/t^n$ is non-decreasing on $(0, 2r_0)$. Therefore 
     \begin{align*}
      \frac{e^{\Lambda_1 2r_0}}{\omega_n (2r_0)^n}\|V_0\|(\BB_{2r_0}(x))\geq 
      &\ \frac{e^{\Lambda_1\cdot (2t_1)}}{\omega_n (2t_1)^n}\|V_0\|(\BB_{2t_1}(x)) \\
      = &\ \frac{e^{\Lambda_1\cdot (2t_1)}}{\omega_n (2t_1)^n}\cdot n_1\|S_1\|(\BB_{2t_1}(x)) + \frac{e^{\Lambda_1\cdot (2t_1)}}{\omega_n (2t_1)^n}\|\hat{V}_1\|(\BB_{2t_1}(x)) \\
      \geq &\ n_1\cdot 2^{-n} + \frac{e^{\Lambda_1\cdot (2t_2)}}{\omega_n (2t_2)^n}\|\hat{V}_1\|(\BB_{2t_2}(x))\\
      = &\ n_1\cdot 2^{-n} + \frac{e^{\Lambda_1\cdot (2t_2)}}{\omega_n (2t_2)^n}\cdot n_2\|S_2\|(\BB_{2t_2}(x)) + \frac{e^{\Lambda_1\cdot (2t_2)}}{\omega_n (2t_2)^n}\|\hat{V}_2\|(\BB_{2t_2}(x)) \\
      \geq &\ (n_1+n_2)\cdot 2^{-n} + \frac{e^{\Lambda_1\cdot (2t_3)}}{\omega_n (2t_3)^n}\|\hat{V}_2\|(\BB_{2t_3}(x)) \\
      &\ ... \\
      \geq &\ (\sum_{i\geq 1} n_i)\cdot 2^{-n} + \theta^n(x, \|V_0\|)    
     \end{align*}
     where $\hat{V}_i$ is a constrained stationary integral varifold in $C_{2t_i}(x)$ defined inductively by \[
      \hat{V}_0 := V_0;\ \ \ \hat{V}_i := (\hat{V}_{i-1} - n_i|S_i|)\llcorner C_{2t_i}(x)     \]
     Hence, the lemma follows from $\|V_0\|(U)\leq \|V\|(U)\leq \Lambda$.
    \end{proof}
    
    Let $\bar{m}:=\sup_{x\in \cB_{r_0}(y)} m(x)$. WLOG $\bar{m}>0$. We are going to define inductively the functions $u_i$ in theorem \ref{Thm_Strc thm_Main}. Let  
    \begin{enumerate}
    \item[•] $\bar{V}_0:= V_0\llcorner C_{r_0}(y)$;
    \item[•] For $1\leq i\leq \bar{m}$, $u_i(x):= \sup\{t: (x, t)\in spt(\bar{V}_{i-1}) \}$, $x\in \cB_{r_0}(y)$;
    \item[•] For $1\leq i\leq \bar{m}$, $\bar{V}_i:= \bar{V}_{i-1} - |graph_{\partial M} (u_i)|$;
    \item[•] For $1\leq i\leq \bar{m}$, $x\in \cB_{r_0}(y)$,  \begin{align*}
     m_i(x):= \begin{cases} 0\ &\ \text{if }\bar{V}_i = 0 \\
      \sum_{p=(x, t)\in spt(\bar{V}_i)} \theta^n(p, \|\bar{V}_i\|)\ &\ \text{if }\bar{V}_i\neq 0
     \end{cases}
    \end{align*}  
    \end{enumerate}
    where we set $\sup \emptyset = -\infty$. The following lemma guarantees they are well-defined. (We set $m_0(x):=m(x)$.)
    \begin{Lem} \label{Lem_Str thm_Induct lem of loc graph}
     For $u_i$, $\bar{V}_i$ defined above, $1\leq i\leq \bar{m}$,
     \begin{enumerate}
     \item[(1)] $u_i$ are nonnegative $C^1$ functions on $\cB_{r_0}(y)$, and is $C^{\infty}$ in $\{x: u_i(x)>0\}$.
     \item[(2)] $\bar{V}_i$ are integral varifolds in $C_{r_0}(y)$ such that $\bar{V}_i\llcorner C^+_{r_0}(y)$ are stable minimal hypersurfaces with multiplicity supported in $spt(V_0)$.
     \item[(3)] $\forall x\in spt(\bar{V}_i)\cap \cB_{r_0}(y)$, the tangent varifold of $\bar{V}_i$ at $x$ is $k|T_x \partial M|$ for some $k\in \NN$.
     \item[(4)] $m_i(x) = m_{i-1}(x) - 1$, $\forall x\in\cB_{r_0}(y)$.
     \end{enumerate}  
    \end{Lem}
    \begin{Rem}
     By (4) of lemma \ref{Lem_Str thm_Induct lem of loc graph} and that $m_i\geq 0$, $m_{\bar{m}}\equiv 0$ and $V_0 = \sum_{i=1}^{\bar{m}} |graph_{\partial M}(u_i)|$. Together with (1), this proves the first part of theorem \ref{Thm_Strc thm_Main}. 
     
     If further, $u_j\in C^{1,1}(\cB_{r_0}(y))$, then (denote $|graph_{\partial M}(u_j)|$ by $\Gamma$ for simplicity) the mean curvature of $\Gamma$ is in $L^{\infty}(\ , \|\Gamma\|)$ satisfying \[
      \vec{H}_{\Gamma} = \vec{H}_{\partial M}\ \ \ \scH^n\text{-a.e. on }\{u_i = 0\}    \]
     Combined with lemma \ref{Lem_Constr stny_Maxm Principle}, we see $\vec{H}_{\Gamma} \cdot \nu_{\partial M} \leq 0$, hence $\Gamma$ is constrained stationary in $C_{r_0}(y)$ along $\partial M$. This proves the second part of theorem \ref{Thm_Strc thm_Main}.
    \end{Rem}

    \begin{proof}[Proof of lemma \ref{Lem_Str thm_Induct lem of loc graph}]:
     It's sufficient to prove (1)-(4) inductively in $i$ assuming they hold for $1\leq l\leq i-1$.\\

     (1) $u_i$ is clearly upper semi-continuous by definition; While by corollary \ref{Cor_Strc thm_Loc Graph}, together with the assumption that $spt(\bar{V}_{i-1})\subset spt(V_0)$ and $\bar{V}_{i-1}\llcorner C_{r_0}^+(y)$ is a stable minimal hypersurface with multiplicity, $u_i$ is lower semi-continuous, and then smooth, in open subset $\{u_i > 0\}\subset \cB_{r_0}(y)$.
     
     Note also that by assumption (4) for $1\leq l\leq i-1< \bar{m}$, $spt(\bar{V}_{i-1})\neq \emptyset$. Hence $\{u_i\geq 0\} \neq \emptyset$. To prove the continuity of $u_i$, it suffices to verify that $u_i\geq 0$. If not, then by upper semi-continuity of $u_i$, there exists an open ball $\cB\subset \{u_i < 0\}$ and some point $z\in \partial \cB\cap \{u \geq 0\}$. Moreover, by the continuity of $u_i$ on $\{u_i>0\}$, $u_i(z)=0$. By definition, $z\in spt(\bar{V}_i)$, then by the assumption (3) for $i-1$, \[
      (\eta_{z, \lambda})_{\sharp}\bar{V}_i \wsto k|T_z \partial M| \ \ \ \text{ as }\lambda\to 0    \]
     for some $k\geq 1$. On the other hand, by definition, $\cB\times (-r_0, r_0)\cap spt(\bar{V}_i) = \emptyset$, hence the weak limit above should be supported in a half space. Here comes a contradiction.
     
     That $u_i$ is differentiable and $C^1$ on $\cB_{r_0}(y)$ follows from corollary \ref{Cor_Strc thm_Loc Graph} and lemma \ref{Lem_Strc thm_tilt est}.\\
     
     (2) integrality of $\bar{V}_i$ follows from assumption (3) for $i-1$, that $u_i\in C^1(\cB_{r_0}(y))$ from (1) and that $graph_{\partial M}(u_i)\subset spt(\bar{V}_{i-1})$ by definition. Hence $spt(\bar{V}_i)\subset spt(\bar{V}_{i-1})\subset ... \subset spt(V_0)$. Since $V_0\llcorner C^+_{r_0}(y)$ is smooth minimal hypersurface with multiplicity and $\{u_l\}_{1\leq l\leq i}$ are $C^1$, $\bar{V}_i \llcorner C^+_{r_0}(y)$ is also a smooth minimal hypersurface with multiplicity. \\
     
     (3) $\forall x\in spt(\bar{V}_i \cap \cB_{r_0}(y))$ fixed, by assumption (3) for $i-1$ and that $u_i\in C^1(\cB_{r_0}(y))$, \[
      (\eta_{x, \lambda})_{\sharp}\bar{V}_i = (\eta_{x, \lambda})_{\sharp} \bar{V}_{i-1} - (\eta_{x, \lambda})_{\sharp }|graph_{\partial M}(u_i)|\ \wsto k|T_{p}\partial M|\ \ \ \text{ as }\lambda\to 0    \]
     for some $k\in \ZZ$. Thus, to prove (3) for $i$, it suffices to verify that $\theta^n(x, \bar{V}_i)>0$. The proof is similar to the one of corollary \ref{Cor_Constr stny_Bdy density lbd}. The only difference is that it's unclear whether $\bar{V}_i$ is constrained stationary in $C_{r_0}(y)$ or not, so lemma \ref{Lem_Constr stny_Bdy Mon} does not apply directly. In what follows, we shall establish a quasi-boundary monotonicity for $\bar{V}_i$.\\
     \textbf{Claim}: $\forall \epsilon >0$, there's an $r_1\in (0, r_0)$ such that, for every $0<\sigma<\eta<r_1$ and every $x\in \cB_{r_0}(y)$ such that $\BB_{2\eta}^L(x)\subset C_{r_0}(y)$, 
     \begin{align}
      \frac{e^{\Lambda \eta}}{\omega_n \eta^n}\|\bar{V}_i\|(\BB^L_{\eta}(x)) 
      \geq \frac{e^{\Lambda \sigma}}{\omega_n \sigma^n}\|\bar{V}_i\|(\BB^L_{\sigma}(x)) -2\bar{m}^2\epsilon   \label{Str thm_quasi-bdy mon}
     \end{align}
     Notice that by taking $\epsilon$ sufficiently small in (\ref{Str thm_quasi-bdy mon}), the proof of corollary \ref{Cor_Constr stny_Bdy density lbd} generalized direct to prove the density lower bound at every point in $spt(\bar{V}_i)$.\\
     \textbf{Proof of the claim}:
     First notice that $\exists\ r_2>0$ and $k>4$ such that, for every $0<r<r_2$, every $\cB_r(z)\subset \cB_{r_0}(y)$ and every $v\in C^1(\cB_{r}(z))$ satisfying \[
      \frac{1}{r}|v| + |\nabla_{\partial M} v| \leq \frac{1}{k}\ \ \ \text{ on }\cB_r(z),   \]
     we have 
     \begin{align}
      {\big|}\frac{e^{\Lambda r}}{\omega_n r^n}\|graph_{\partial M}(v)\|(\cB_r(z)) - 1 {\big|}\leq \epsilon  \label{Str thm_alm unit vol}     
     \end{align}
     where $\Lambda$ is in lemma \ref{Lem_Constr stny_Bdy Mon}.
     
     Now take $r_1:= min\{r_2, \rho(1/(40k^2)), 1/k\}/10$, where $\rho$ is specified in corollary \ref{Cor_Strc thm_Loc Graph}. For every $0 <r <r_1$ and every $\cB_r(x)\subset \cB_{r_0}(y)$, by our choice of $r_1$ and corollary \ref{Cor_Strc thm_Loc Graph}, $\forall 1\leq l\leq i$, either 
     \begin{align}
      \frac{1}{r}|u_l| < \frac{1}{2k},\ \ \ |\nabla_{\partial M} u_l| \leq \frac{1}{2k}\ \ \ \text{ on }\cB_{2r}(x)  \label{Str thm_Alm flat}
     \end{align}
     or
     \begin{align}
      u_l >0\ \ \ \text{ on }\cB_{2r}(x)  \label{Str thm_glob stny}
     \end{align}
     Note that if $u_l$ satisfies (\ref{Str thm_glob stny}) on $\cB_{2r}(x)$, then it satisfies (\ref{Str thm_glob stny}) on every $\cB_t(x)\subset \cB_{2r}(x)$. Also note that once $u_l>0$ on $\cB_{2r}(x)$, $|graph_{\partial M}(u_l)|$ is a stationary varifold in $\cB_{2r}(x)\times (-r_0, r_0)$ and hence $V_0 - |graph_{\partial M}(u_l)|$ is constrained stationary in $\cB_{2r}(x)\times (-r_0, r_0)$.
     
     Now that $\forall 0<\sigma<\eta<r_1$, $\forall \cB_{2\eta}(x)\subset \cB_{r_0}(y)$, let 
     \begin{enumerate}
      \item[] $I_1:=\{l: 1\leq l\leq i, u_l \text{ satisfies }(\ref{Str thm_Alm flat}) \text{ on }\cB_t(x) \text{ for every }t\in [\sigma, \eta] \}$
      \item[] $I_2:=\{l: 1\leq l\leq i, u_l \text{ satisfies }(\ref{Str thm_glob stny}) \text{ on }\cB_t(x) \text{ for every }t\in [\sigma, \eta] \}$
      \item[] $I_3:=\{1, 2, ..., i \}\setminus (I_1\cup I_2)$
     \end{enumerate}
     And for $l\in I_3$, let $t_l:=\inf\{t\in [\sigma, \eta]: (\ref{Str thm_glob stny})\text{ fails on }\cB_{2t}(x) \}$. After relabeling the index, WLOG $I_3=\{1, 2, ... q\}$ and $t_1\leq t_2\leq ... t_q$, $q\in \NN\cup\{0\}$.  Denote for simplicity, \[
      a(x, t; V):= \frac{e^{\Lambda t}}{\omega_n t^n}\|V\|(\BB_t^L(x)); \ \ \ 
      a_l(x, t):= a(x, t; |graph_{\partial M}(u_l)|)    \]
      and $\Gamma_l:= |graph_{\partial M}(u_l)|$. Then
     \begin{align*}
      a(x, \eta; \bar{V}_i) 
      = &\ a(x, \eta; V_0-\sum_{l\in I_2}\Gamma_l) - \sum_{l=1}^q a_l(x, \eta) - \sum_{k\in I_1} a_k(x, \eta) \\
      \geq &\ a(x, t_q; V_0-\sum_{l\in I_2}\Gamma_l) - \sum_{l=1}^q (a_l(x, t_q)+ 2\epsilon) - \sum_{k\in I_1} (a_k(x, \sigma)+2\epsilon) \\
      \geq &\ a(x, t_q; V_0-\sum_{l\in I_2}\Gamma_l - \Gamma_q) - \sum_{l=1}^{q-1} a_l(x, t_q) - \sum_{k\in I_1} a_k(x, \sigma) - 2\bar{m}\epsilon \\
      \geq &\ a(x, t_{q-1}; V_0-\sum_{l\in I_2}\Gamma_l - \Gamma_q) - \sum_{l=1}^{q-1} (a_l(x, t_{q-1}) + 2\epsilon) - \sum_{k\in I_1} a_k(x, \sigma) - 2\bar{m}\epsilon \\
      \geq &\ a(x, t_{q-1}; V_0-\sum_{l\in I_2}\Gamma_l - \Gamma_q-\Gamma_{q-1}) - \sum_{l=1}^{q-2} a_l(x, t_{q-1}) - \sum_{k\in I_1} a_k(x, \sigma) - 4\bar{m}\epsilon \\
      &\ ...\ ... \\
      \geq &\ a(x, t_1; V_0-\sum_{l\in I_2}\Gamma_l - \sum_{l=1}^q \Gamma_l) -0 - \sum_{k\in I_1} a_k(x, \sigma) - 2q\bar{m}\epsilon \\
      \geq &\ a(x, \sigma; V_0-\sum_{l\in I_2\cup I_3}\Gamma_l) -\sum_{k\in I_1} a_k(x, \sigma) - 2q\bar{m}\epsilon \\
      \geq &\ a(x, \sigma; \bar{V}_i) - 2\bar{m}^2\epsilon
     \end{align*}
     where the first inequality follows from (\ref{Str thm_alm unit vol}), lemma \ref{Lem_Constr stny_Bdy Mon} and that $V_0-\sum_{l\in I_2}\Gamma_l$ is constrained stationary in $\cB_{\eta}(x)\times (-r_0, r_0)$; the third inequality follows form (\ref{Str thm_alm unit vol}), lemma \ref{Lem_Constr stny_Bdy Mon} and that $V_0-\sum_{l\in I_2}\Gamma_l - \Gamma_q $ is stationary in $\cB_{2t_q}(x)\times (-r_0, r_0)$ and so on. This completes the proof of claim.\\
     
     (4) follows immediately from (3) and corollary \ref{Cor_Strc thm_Loc Graph}. 
    \end{proof}
    \begin{Rem}
     From the proof we can also see that $r_0 = r_0(M, U, \cK, \Lambda)$ can be chose to continuously depend on $(M, \partial M)$ in $C^2$ topology.
    \end{Rem}

   \section{Min-max with Obstacle} \label{Section_Minmax}
    \subsection{Almost minimizer and regularity} \label{Subsection_Alm minimizer}
     The notion of almost minimizing varifolds is introduced in \cite{Pitts81} to prove the regularity of stationary varifolds obtained from min-max construction. We employ the similar notions here. The only difference is that everything is supported in an ambient manifold with boundary, making the regularity result more subtle near the boundary. $M$ will be as in subsection \ref{Subsection_Notations}, $U\subset M$ be a relative open subset.
     \begin{Def}
      Call $\{T^{(l)}\}_{l=0}^q\subset \bfI_n(U)$ a \textbf{$\delta$-sequence} in $U$ start at $T^{(0)}$ if $\forall 1\leq l\leq q$,
      \begin{enumerate}
       \item[•] $spt(T^{(l)}-T^{(0)})\subset\subset U$
       \item[•] $\bfM_U(T^{(l)}) \leq \bfM_U(T^{(0)}) + \delta$
       \item[•] $\bfM_U(T^{(l)} - T^{(l-1)}) \leq \delta$
      \end{enumerate}
      Call $V\in \cV_n(M)$ \textbf{almost minimizing} in $U$, if there exist $\epsilon_i\to 0$, $\delta_i\to 0$, $T_i\in \bfI_n(M)$ with $|T_i|\to V$ such that for every $i$ and every $\delta_i$-sequence $\{T_i^{(l)}\}_{i=0}^{q}$ in $U$ start at $T_i$, we have \[
       \bfM_U(T_i^{(q)}) \geq \bfM_U(T_i) - \epsilon_i    \]
     \end{Def}
     Note that by definition, if $V$ is not almost minimizing in $U$, then $\exists\ \epsilon>0$ such that $\forall T\in \bfI_n(M)$ with $\bfF(|T|, V)< \epsilon$ and $\forall \delta>0$, there exists some $\delta$-sequence $\{T^{(l)}\}_{l=0}^q$ start at $T$ with \[
      \bfM_U(T^{(q)}) \leq \bfM_U(T) - \epsilon     \]
     
     \begin{Thm} \label{Thm_Minmax_Reg alm min}
      Let $V\in \cV_n(M)$ be  almost minimizing in $U$. Then $V$ is a constrained embedded stable minimal hypersurface in $U$ with optimal regularity.
     \end{Thm}
     \begin{proof}
      First see that by \cite{Pitts81}, $V\llcorner Int(U)$ is a stable minimal hypersurface with optimal regularity; moreover, a similar argument in \cite[3.3]{Pitts81} yields that $V$ is constrained stationary in $U$. Now we focus on the behavior of $V$ near $\partial M$. The proof of integrality of $V$ is almost the same as in \cite[Chapter 3]{Pitts81}, with aid of the lemma in subsection \ref{Subsection_Constr stny vfd}. We briefly go through the proof.\\

\noindent   \textbf{$\bullet$ Replacement}.  $\forall U'\subset\subset U$ be a subdomain. Since $V$ is almost minimizing in $U$, let $\epsilon_i\to 0$, $\delta_i\to 0$ and $|T_i|\to V$ be as in the definition. For each $i$, let $T_i^*$ be a mass minimizer among \[
       \scE(T_i, Clos(U')):= \{T^{(q)}: \exists\ \delta_i\text{-sequence }\{T^{(l)}\}_{l=0}^q \text{ start at }T_i \text{ such that }spt(T^{(l)}-T_i)\subset Clos(U')\}   \]
      The existence of such $T_i^*\in \scE(T_i, Clos(U'))$ follows from Federer-Fleming compactness \cite{FedererFleming60} and discretization-interpolation theorem \cite[13.1, 14.1]{MarquesNeves14}. 
      
      Moreover, $T_i^*$ is locally constrained minimizing in $U'$ and satisfies the mass bound \[
       \|T_i\|(U')-\epsilon_i \leq \|T_i^*\|(U') \leq \|T_i\|(U')     \] 
      Hence, by theorem \ref{Thm_Survey Obst_Reg for para}, $|T_i^*|\llcorner U'$ is a constrained embedded stable minimal hypersurface in $U'$. Then by corollary \ref{Cor_Str thm_Cptness of c-stable min}, when $i\to \infty$, $|T_i^*|\to V^*\in \cV_n(U)$ satisfying 
      \begin{enumerate}
       \item[(a)] $V^*$ is almost minimizing in $U$ and $V^*\llcorner U'$ is a constrained embedded stable minimal hypersurface;
       \item[(b)] $V^*\llcorner Clos(U')^c = V\llcorner Clos(U')^c$;
       \item[(c)] $\|V^*\|(U) = \|V\|(U)$
      \end{enumerate} 
      Those $V^*\in \cV_n(U)$ satisfying (a)-(c) above will be called \textbf{replacements} of $V$ in $U'$.\\
      
\noindent   \textbf{$\bullet$ Rectifiability}.  Thanks to Allard rectifiability theorem \cite[Chapter 8, 5.5]{Simon83_GMT} and lemma \ref{Lem_Constr stny_Loc bd 1st var}, it suffices to show that $\theta^n(x, \|V\|)>0$ $\|V\|$-a.e. $x\in \partial M \cap U$. $\forall x\in spt(V)\cap U\cap \partial M$, let $r_j\in (0, dist_M(x, M\setminus U)/2)$, $r_j\to 0$; $V_j^*$ be a replacement of $V$ in $ A_{r_j, 2r_j}(x)$. Then by lemma \ref{Lem_Constr stny_Bdy Mon}, $\|V_j^*\|(A_{5r_j/4, 7r_j/4})>0$ for $j>>1$. By the same trick as in corollary \ref{Cor_Constr stny_Bdy density lbd}, $\|V_j^*\|(A_{r_j, 2r_j}) \geq c(n, M)r_j^n$ for some $c(n, M)>0$. Hence, $\theta^n(x, \|V\|) \geq c'(n, M)>0$.\\

\noindent    \textbf{$\bullet$ Integrality}. For $V_j^*$ constructed above, by lemma \ref{Lem_Constr stny_Cptness} and \ref{Lem_Constr stny_Bdy Mon}, up to a subsequence, $(\eta_{x, r_j})_{\sharp} V \to P$, $(\eta_{x, r_j})_{\sharp} V_j^* \to Q$, where $P, Q  \in \cV_n(\RR^{n+1}_+)$ are constrained stationary along $\RR^n\times \{0\}$. ( We identify $(T^+_x M, T_x\partial M)$ with $(\RR^{n+1}_+, \RR^n\times \{0\})$ for simplicity. ) Moreover, by definition of replacements, $Q\llcorner \AAa^{n+1}_{1,2}(0)$ is constrained embedded stable minimal hypersurface, hence stable minimal hypersurface by lemma \ref{Lem_Constr stny_Maxm Principle}; And $P = Q$ outside $Clos(\AAa^{n+1}_{1,2}(0))$. By lemma \ref{Lem_Constr stny_Bdy Mon}, $P, Q$ are cones and hence $P = Q$ is a stable minimal cone with optimal regularity supported in $\RR^{n+1}_+$. By strong maximum principle, $P = Q = m|\RR^n\times \{0\}|$ for some $m\in \NN$. Therefore $\theta^n(x, \|V\|) = m \in \NN$. This proves $V\in \cI\cV_n(U)$.\\

      Now by theorem \ref{Thm_Strc thm_Main}, $V$ is locally multiple $C^1$ graphs near $\partial M$. To get $C^{1,1}$ regularity of the graphical function, we use again the replacements. In what follows, we keep working under Fermi coordinates.

      Let $x\in \partial M\cap U$, $r_0=r_0(M, U, \{x\}, \|V\|(U))>0$ be in theorem \ref{Thm_Strc thm_Main}; $U_x \subset C_{r_0}(x)$ be a neighborhood of $x$ such that $V\llcorner U_x = \sum_{i=1}^m |graph_{\partial M}(u_i)|$, where $u_1\geq u_2\geq ... \geq u_m \in C^1(U_x\cap \partial M, \RR_+)$, $|\nabla u_i|\leq 1/2$ and $\{u_1= 0\} \neq \emptyset$. For any $p = (y, t)\in graph_{\partial M}(u_i)\cap Int(M)$,  let $V^*$ be a replacement of $V$ in $U_x\setminus Clos(B_{t/2}(p))$. By theorem \ref{Thm_Strc thm_Main}, $V^*\llcorner U_x = |graph_{\partial M}(u^*)| + \hat{V}^*$ for some $\hat{V}^*\in\cI\cV_n(U_x)$ and $u^* = u_i$ on the connected component of $\{u_i>0\}$ containing $y$ by unique continuation; Since $V^*$ is a constrained embedded stable minimal hypersurface in $U_x$, by theorem \ref{Thm_Survey Obst_Reg for non-para}, $u^* \in C^{1,1}(\partial M\cap U)$ and for every $\cK\subset \partial M\cap U$ compact, \[
       \|u^*\|_{C^{1,1}(\cK)} \leq C(M, r_0, \cK)     \]
      By applying the replacement on every connected component of $spt(V)\cap Int(U_x)$, we see $u_i\in C^{1,1}_{loc}(\partial M\cap U_x)$. By the arbitrariness of $i$ and $x$ and second part of theorem \ref{Thm_Strc thm_Main}, $V$ is a constrained embedded stable minimal hypersurface in $U$.
     \end{proof}

    \subsection{Min-max construction} \label{Subsection_Minmax contruction}
     In this subsection, we describe the min-max construction with obstacle following the idea from \cite{Pitts81}. The notions coincide with \cite[section 7,8]{MarquesNeves14} for the cell complexes, subcomplex and homotopy relations. We also use the notions of relative homotopy classes in \cite{Zhou19_multi1}. 
    
     Suppose $m\geq 1$, $M$ be described in subsection \ref{Subsection_Notations}. Let $(X, Z)$ be a pair of cubical complex in $I^m$ with $Z\subset X$. Fix $\Phi_0: X \to (\cZ_n(M), \bfF)$ be a continuous map (i.e. continuous in $\bfF$-metric). Let $\Pi$ be the space of sequence of continuous maps $S = \{\Phi_i: X\to (\cZ_n(M), \bfF)\}_{i\geq 1}$ satisfying \[
     \sup \{\bfF(\Phi_i(x), \Phi_0(x)): x\in Z \} \to 0 \text{ as } i\to \infty   \]
     And define the width 
     \begin{align*}
      \bfL(S) & := \limsup_{i\to \infty} \sup_{x\in X} \bfM(\Phi_i(x)) \\ 
      \bfL(\Pi) & := \inf_{S\in \Pi} \bfL(S) \\
      \bfL_{\Phi_0, Z} & := \sup_{x\in Z} \bfM(\Phi_0(x))
     \end{align*}
     
     The goal of this section is to prove 
     \begin{Thm} \label{Thm_Minmax_Exist constr min}
      Let $(M, g)$ be a compact $n+1$ manifold with boundary, $\Phi_0: X\to (\cZ_n(M), \bfF)$ be a continuous map as above. If $\bfL(\Pi)>\bfL_{\Phi_0, Z}$, then there's a constrained embedded minimal hypersurface $V\in \cI\cV_n(M)$ with optimal regularity such that $\|V\|(M)= \bfL(\Pi)$.
     \end{Thm}
     With aid of theorem \ref{Thm_Minmax_Reg alm min}, the proof is almost the same as in \cite[Chapter 4]{Pitts81}. We briefly go through the major part of the argument.\\

\noindent   \textbf{$\bullet$ Minimizing sequence and pull tight}. 
      By a diagonalization process, there exists some $S\in \Pi $ such that $\bfL(S) = \bfL(\Pi)$. Call such $S$ a \textbf{minimizing sequence}.
      
      For a minimizing sequence $S = \{\Phi\}_i\in \Pi$, denote \[
       \bfC(S):=\{V\in \cV_n(M): V_n = \lim_{j\to \infty} |\Phi_{i_j}(x_j)|, \ \|V\|(M)= \bfL(S) \}   \]
      called the \textbf{critical set} of $S$. Following \cite[4.3]{Pitts81} and \cite{ZhouZhu18}, we can construct $H: [0,1]\times \cZ_n(M)\cap \{\|T\|(M)\leq 2\bfL(\Pi)\} \to \cZ_n(M)\cap \{\|T\|(M)\leq 2\bfL(\Pi)\}$ continuous in $\bfF$-metric such that 
      \begin{enumerate}
      \item[(i)] $H(0, \text{-}) = id$;
      \item[(ii)] $H(t, T) = T$, $\forall t\in [0, 1]$, $\forall |T|\in \cV^{\infty}:=\{\text{constrained stationary varifolds in }M\}\cup \{|\Phi(x)|: x\in Z\}$;
      \item[(iii)] There's an $L: [0, +\infty)\to [0, +\infty)$ continuous, $L(0)=0$, $L(\tau)>0$ if $\tau>0$ such that \[
        \|H(1, T)\|(M) \leq \|T\|(M) - \bfL(\bfF(|T|, \cV^{\infty}))     \]
      \item[(iv)] $\forall \epsilon>0$, $\exists\ \delta>0$ such that if $\bfF(|T|, |\Phi_0(x)|)\leq \delta$ for some $x\in Z$, then $\bfF(H(t, T), |\Phi_0(x)|)\leq \epsilon$ for all $t\in [0, 1]$.
      \end{enumerate}
      Hence, if $S\in \Pi$ is a minimizing sequence, then $S':=\{H(1, \Phi_i)\}_{i\geq 1}\in \Pi$ is also a minimizing sequence  such that $\emptyset\neq \bfC(S')\subset \bfC(S)$ consists of constrained stationary varifolds in $M$. Call such $S'\in \Pi$ a \textbf{pulled tight sequence}.\\
      
\noindent   \textbf{$\bullet$ Discretization and interpolation}.
      Given a pulled-tight sequence $S=\{\Phi_i\}_{i\geq 1}\in \Pi$, by the interpolation theorem \cite[13.1]{MarquesNeves14} and diagonalization process, there exist $k_i\nearrow \infty$, $\delta_i\to 0$ and $\varphi_i: X(k_i) \to \cZ_n(M)$ such that 
      \begin{enumerate}
       \item[(i)] $\sup \{\bfF(\varphi_i(x), \Phi_i(x)): x\in X(k_i)_0\} \leq \delta_i$;
       \item[(ii)] $\sup \{\bfF(\varphi_i(x), \varphi_i(y)): x, y\in \alpha \cap X \text{ for some }\alpha\in I(m, k_i)_m\} \leq \delta_i$
       \item[(iii)] $\textbf{f}(\varphi_i) \to 0$ as $i\to \infty$;
       \item[(iv)] $\bfL(\{\varphi_i\}_i):= \limsup_{i\to \infty} \sup_{x\in X(k_i)_0} \|\varphi_i(x)\|(M) = \bfL(\Pi)$.
      \end{enumerate}
      In particular, \[
       \bfC(\{\varphi_i\}_{i\geq 1}):= \{V : V = \lim_{j\to \infty} |\varphi_{i_j}(x_j)| \text{ for some }i_j\to \infty \text{ and }x_j\in X(k_{i_j})_0;\ \|V\|(M) = \bfL(\Pi) \} \subset \bfC(S)  \]
       
      Conversely, given a sequence of $k_i\nearrow \infty$, $\delta_i\to 0$ and $\varphi_i : X(k_i)_0 \to \cZ_n(M)$ satisfying (ii), (iii) above and 
      \begin{enumerate}
       \item[(i')] $\sup \{\bfF(\varphi_i(x), \Phi_0(x)): x\in Z(k_i)_0\} \leq \delta_i$.
      \end{enumerate}
      By the interpolation theorem in \cite[section 6]{Almgren62}, (see also \cite[14.1]{MarquesNeves14}), there exists $S = \{\Phi_i\}_{i\geq 1} \in \Pi$ such that \[
       \bfL(S) = \bfL(\{\varphi_i\}_{i\geq 1})   \]
      
\noindent   \textbf{$\bullet$ Pitts' deformation}.
      Let $\{\varphi_i: X(k_i)_0 \to \cZ_n(M)\}_{i\geq 1}$ be satisfying (i'), (ii), (iii) above and such that \[
       \bfL(\{\varphi_i\}_{i\geq 1}) > \bfL_{\Phi_0, Z}, \ \ \ \bfC(\{\varphi_i\}_{i\geq 1}) \subset \{\text{constrained stationary varifolds in }M\}   \]
      Assuming $\forall V\in \bfC(\{\varphi_i\}_i)$, $\exists\ x\in M$ such that $\forall r>0$, $\exists s\in (0, r)$ with $V$ not almost minimizing in $A_{s,r}(x)$, \cite{Pitts81} introduced a deformation process to construct $\{\varphi'_i : X(k_i')_0\to \cZ_n(M)\}$ satisfying (i'), (ii), (iii) and \[
       \bfL(\{\varphi'_i\}_{i\geq 1}) < \bfL(\{\varphi_i\}_{i\geq 1})     \] 
      Applying this to the interpolation theorem above, one see that given a pulled-tight minimizing sequence $S\in \Pi$, $\exists\ V\in \bfC(S)$ such that $\forall x\in M$, $\exists r_x>0$ and $V$ being almost minimizing $A_{s, r}(x)$, $\forall s\in (0, r)$. By theorem \ref{Thm_Minmax_Reg alm min} and \ref{Thm_Survey Obst_Reg for non-para}, $V$ is a constrained embedded minimal hypersurface. This complete the proof of theorem \ref{Thm_Minmax_Exist constr min}.
      
      \begin{Rem}
       By Pitts' combinatorial argument \cite[4.9, 4.10]{Pitts81}, there's a constant $N(m)\in \NN$ such that in theorem \ref{Thm_Minmax_Exist constr min}, we can choose $V$ such that if $p\in M$ and $\{A_j:= A_{r_j, s_j}(p)\}_{j=1}^{N(m)}$ is a family of annuli in $\tilde{M}$ such that $s_j>r_j> 4s_{j+1}$ for every $j$, then $V$ is almost minimizing in at lease one of $A_j\cap M$. Hence in particular $V$ is a constrained embedded \textbf{stable} minimal hypersurface in that $A_j\cap M$.
      \end{Rem}

   \section{A local rigidity for constrained embedded minimal hypersurfaces} \label{Section_Loc rigid}
    We close this article by pointing out the following local rigidity theorem for constrained embedded minimal hypersurface. 
    Let $(N, g)$ be a closed $n+1$ manifold; $(\Sigma, \nu)\subset N$ be a connected, two-sided, closed smooth minimal hypersurface; $r_{\Sigma}$ be the injective radius of $\Sigma$ in $N$, i.e. \[
     r_{\Sigma}:= \sup\{r: \Phi: \Sigma\times (-r, r)\to N,\ (x, t)\mapsto \exp^N_x(t\nu(x)) \text{ is a diffeomorphism onto its image}\}     \]
    Work under Fermi coordinates in $\Phi(\Sigma\times (-r_{\Sigma}, r_{\Sigma}))$. For $K>1$ and $w^{\pm}\in C^\infty(\Sigma, (0, r_{\Sigma}))$, call the domain \[
     \scN_{w^{\pm}}:= \{(x, t)\in N: x\in \Sigma, -w^- \leq t\leq w^+\}   \]
    a \textbf{$K$-uniform} neighborhood of $\Sigma$, if
    \begin{enumerate}
     \item[(1)] $\sup_{x\in\Sigma, i\in \pm} w^i(x) \leq K\inf_{x\in\Sigma, i\in \pm} w^i(x)$;
     \item[(2)] The mean curvature $H_{\pm}$ of $graph_{\Sigma}(\pm w^{\pm})$ satisfies \[  |H_{\pm}|\leq K\sup_{\Sigma} w^{\pm}     \]
    \end{enumerate}
    Note that for a smooth minimal hypersurface as above, the $\delta$-neighborhood $\scN_{\delta}:= \{p\in N: dist_N(p, \Sigma)\leq \delta\}$ and the neighborhoods $\scN_{\delta w_1}:= \{(x, t): x\in \Sigma, |t|\leq \delta w_1(x)\}$ given by first eigenfunction $w_1$ of Jacobi operator of $\Sigma$ are both $K$-uniform neighborhood of $\Sigma$ for every sufficiently small $\delta$ and $K$ independent of $\delta$.
    
    \begin{Thm} \label{Thm_Loc rigid_Main thm}
     Let $\Sigma\subset N$ be a non-degenerate closed minimal hypersurface as above; $K>1$; $\epsilon\in (0,1)$. Then $\exists\ s_0 = s_0(K, \epsilon, \Sigma, N)>0$ such that if $\scN = \scN_{w^{\pm}}$ is a $K$-uniform neighborhood of $\Sigma$ and $\sup_{\Sigma} |w^{\pm}|\leq s_0$, and $V$ is a constrained embedded minimal hypersurface in $\scN$ with \[
      \scH^n(\Sigma)\leq \|V\|(N) \leq (2-\epsilon)\scH^n(\Sigma)     \]
     then $V = |\Sigma|$. 
    \end{Thm}
    We emphasis that having the same mass is crucial for such a theorem to be true. In fact, for an unstable minimal hypersurface $\Sigma$, let $w^+ = w^- = \delta w_1$, where $w_1$ be the first eigenfunction of Jacobi operator, $\delta<<1$, then each component of $\partial \scN_{w^{\pm}}$ is constrained embedded minimal in $\scN_{w^{\pm}}$.
    
    \begin{proof}
     We prove by contradiction. First recall some basic facts on minimal surface equation on $\Sigma$.
     Let $F=F(x, z, p)$, $x\in \Sigma$, $z\in (-r_{\Sigma}, r_{\Sigma}), p\in T_x\Sigma$ be the area integrand for graph over $\Sigma$, i.e. \[
      \scH^n(graph_{\Sigma}(\phi)) = \int_{\Sigma} F(x, \phi(x), \nabla\phi(x))\ dx\ \ \ \ \forall \phi\in C^1(\Sigma, (-r_{\Sigma}, r_{\Sigma}))    \]
     Here are some basic propositions we shall use later,
     \begin{enumerate}
      \item[(a)] $F(x, 0 ,0) = 1$; $F_z(x, 0, 0):=\partial_z F(x, 0, 0) = 0$,  $F_p := \partial_p F(x, 0, 0) = 0$; \\
      $F_{pp}(x, 0, 0)^{ij} = \delta^{ij}$, $F_{zp}(x, 0, 0 )=F_{pz}(x, 0, 0) = 0$, $F_{zz}(x, 0, 0) = - |A_{\Sigma}|^2-Ric_N(\nu_{\Sigma}, \nu_{\Sigma})$.
      \item[(b)] $\|F\|_{C^3}\leq C(\Sigma, N)$.
      \item[(c)] If further $\phi\in C^1\cap W^{2,2}(\Sigma)$, then \[
        \scM \phi := -div F_p(x, \phi, \nabla\phi) + F_z(x, \phi, \nabla \phi)   = \frac{1}{\sqrt{1+|\nabla \phi|^2}}H \]
       where $H$ is the mean curvature of $graph_{\Sigma}(\phi)$ with normal field having positive inner product with $\nu$. 
      \item[(d)] $L_{\Sigma}\phi:=\frac{d}{ds}\big|_{s=0} \scM(s\phi) = -(\Delta +|A_{\Sigma}|^2 +Ric_N(\nu, \nu))\phi$ \ be the Jacobi operator of $\Sigma$.
     \end{enumerate}
     
     Suppose otherwise, $K>1$ and there are $K$-uniform neighborhood $\scN_j = \scN_{w_j^{\pm}}$ of $\Sigma$ with $\sup |w_j^{\pm}|\to 0$ and $|\Sigma|\neq |\Gamma_j|\in \cI\cV_n(\scN_j)$ constrained embedded minimal hypersurfaces in $\scN_j$ with $\|\Gamma_j\|(N) \in [\|\Sigma\|(N), (2-\epsilon)\|\Sigma(N)\|]$. The strategy is to show that these $\{\Gamma_j\}$ induces a nontrivial Jacobi field on $\Sigma$, which contradicts to the non-degeneracy assumption.
     
     Recall by Remark \ref{Rem_Constr emb_Bd mean curv} (1), $\Gamma_j$ has mean curvature uniformly bounded by the mean curvature of $\partial \scN_j$, hence tends to $0$ in $L^{\infty}$. By Allard compactness theorem \cite{Allard72, Simon83_GMT}, up to a subsequence, $\Gamma_j\to \Gamma_{\infty}$ for some stationary integral varifold $\Gamma_{\infty}\in \cI\cV_n(N)$ supported in $\Sigma$ and having mass between $\|\Sigma\|(N)$ and $(2-\epsilon)\|\Sigma\|(N)$. Hence $\Gamma_{\infty} = |\Sigma|$, and by Allard regularity \cite{Allard72, Simon83_GMT}, for $j>>1$, there exist $u_j\in C^{1, \alpha}\cap W^{2,2} (\Sigma)$ such that $\Gamma_j = |graph_{\Sigma}(u_j)|$, $u_j\to 0$ in $C^{\alpha}$, not identically $0$ and 
     \begin{align}
     \begin{split}
      \scM u_j &= 0 \ \ \ \text{ on }\{-w_j^- < u_j < w_j^+\} \\
      u_j\cdot\scM u_j &\leq 0 \ \ \ \text{ on }\Sigma \\
      |\scM u_j| &\leq \sup_{i\in \pm, x\in \Sigma} Kw_j^i(x)\ \ \ \text{ on }\Sigma
     \end{split} \label{Loc rigid_Constr min surf equ}
     \end{align}
     where the last two inequality follows from Remark \ref{Rem_Constr emb_Bd mean curv} (1). Note that $\{u_j = \pm w^{\pm}_j\}\neq \emptyset$, otherwise by standard elliptic estimate, $u_j/\|u_j\|_{L^2}$ will $C^{\infty}$-converge to some nontrivial Jacobi field on $\Sigma$. Thus, 
     \begin{align}
      \inf w_j^{\pm}\leq \|u_j\|_{C^0(\Sigma)}\leq \sup w_j^{\pm} \label{Loc rigid_Two side C^0 bd}
     \end{align}
      
     Also  by standard elliptic estimate, \[
      \|u_j\|_{C^{1,\alpha}} + \|u_j\|_{W^{2,2}}\leq C(\Sigma, N)(\|u_j\|_{C^0} + |\scM u_j|_{C^0}) \leq C(\Sigma, N, K)\|u_j\|_{C^0}    \]
     Hence, let $c_j:=\|u_j\|_{C^0}$, up to a subsequence, $\hat{u}_j:= u_j/c_j \to \hat{u}_{\infty} \in C^{1,\alpha}\cap W^{2,2}(\Sigma)$ in $C^1$. Moreover, $\|\hat{u}_{\infty}\|_{C^0}=1$ and by (\ref{Loc rigid_Constr min surf equ}), (\ref{Loc rigid_Two side C^0 bd}) and definition of $K$-uniform domain, 
     \begin{align}
     \begin{split}
      L_{\Sigma} \hat{u}_{\infty} &= 0\ \ \ \text{ on }\{|\hat{u}_{\infty}|<1/K \}  \\
      \hat{u}_{\infty}\cdot L_{\Sigma} \hat{u}_{\infty} &\leq 0\ \ \ \text{ on }\Sigma 
     \end{split} \label{Loc rigid_Jac inequ globally}
     \end{align}
     
     Now we make use of the constraint on mass to show that $\int_{\Sigma} \hat{u}_{\infty}\cdot L_{\Sigma}\hat{u}_{\infty} \geq 0$. Combining this with (\ref{Loc rigid_Jac inequ globally}), we see that $\hat{u}_{\infty}$ is a nontrivial Jacobi field on $\Sigma$, which contradicts to the non-degeneracy.
     
     Observe that 
     \begin{align}
     \begin{split}
      0 & \leq \|\Gamma_j\|(N)-\|\Sigma\|(N) 
       = \int_{\Sigma} F(x, u_j, \nabla u_j)-1 \ dx \\
      & = \int_{\Sigma}(\int_0^1 F_p(x, su_j, s\nabla u_j)\ ds)\cdot \nabla u_j + (\int_0^1 F_z(x, su_j, s\nabla u_j))\cdot u_j\ dx \\
      & =\ \frac{1}{2}\int_{\Sigma} F_p(x, u_j, \nabla u_j)\cdot \nabla u_j + F_z(x, u_j, \nabla u_j)u_j\ dx \ +\  \scR
     \end{split}  \label{Loc rigid_Mass diff}
     \end{align}
     where by (a),(b) and standard elliptic estimate, $|\scR|\leq C(\Sigma, N, K)c_j^3$. Here we use the fact that \[
      \int_0^1 f(s)\ ds - \frac{1}{2}(f(0)+ f(1)) = \frac{1}{2}\int_0^1 f''(s)(s^2-s)\ ds     \]
     Multiply (\ref{Loc rigid_Mass diff}) by $c_j^{-2}$ and let $j\to \infty$ one get, \[
      \int_{\Sigma} \hat{u}_\infty\cdot L_{\Sigma} \hat{u}_\infty  =\int_{\Sigma} |\nabla \hat{u}_\infty|^2 - (|A_{\Sigma}|^2 + Ric_N(\nu, \nu))\hat{u}_{\infty}^2\ dx \geq 0     \]
     which completes the proof.
    \end{proof}

    \begin{Rem}
      An interesting question is when the constrained embedded minimal hypersurface obtained by min-max construction in theorem \ref{Thm_Minmax_Exist constr min} is a minimal hypersurface (i.e. has vanishing mean curvature). When $\partial M$ is minimal, this is clearly true by strong maximum principle; In general, is this true when one assume further that the width is invariant under local perturbation of $\partial M$ in $\tilde{M}$? Is this true under assumption above for generically perturbed boundary? 
    \end{Rem}

\appendix

   \section{Improved regularity for obstacle problems}
    The goal of this section is to prove theorem \ref{Thm_Survey Obst_Reg for non-para} and \ref{Thm_Survey Obst_Reg for para}. First recall that under local coordinate chart $(\cU, x^i)$ of $\partial M$, if $u\in C^1(\cU, (-r_M/2, r_M/2))$, then $\scH^n(graph_{\partial M}(u))= \int_{\cU} F(x, u, du)$, where $0< F(x, z, p)\in C^{\infty}(\cU\times(-r_M, r_M)\times \RR^n)$ satisfies 
    \begin{align}
     \partial^2_{pp}F(x, z, p)^{ij}\xi_i\xi_j & \geq \lambda(x,z, p) |\xi|^2 >0\ \ \ \ \forall \xi\in\RR^n \\
     \|F\|_{C^4(\cU\times [-r_M/2, r_M/2]\times \BB^n_R)} & \leq C(M, g, R)
    \end{align}
    If $u\geq 0$ and $graph_{\partial M}(u)$ is constrained stationary in $M$, then 
    \begin{align}
     \begin{split}
      0\leq & \frac{d}{dt}\Big|_{t=0}\int_{\cU} F(x, u+t\phi, du+td\phi) \ dx \\
      \leq & \int_{\cU} \partial_p F(x,u, du)\cdot d\phi + \partial_z F(x, u, du)\phi\ dx
     \end{split} \label{Appendix_Constr stny}
    \end{align}
    For every $\phi \in C_c^1(\cU)$ such that for $0<t<<1$, $u+t\phi \geq 0$.\\
    
    We shall deal with more general quasi-linear 2nd order elliptic operator of divergence form. For simplicity, write $B_r(x):= \BB^n_r(x)\subset \RR^n$, $B_r := B_r(0)$ and $\Omega_r:= B_r\times [-1, 1]\times B_1$. Let $\lambda, \Lambda>0$, $\alpha\in (0, 1)$ be constant,  $A = A(x, z, p)\in C^3(B_2\times \RR\times \RR^n, \RR^n)$, $B=B(x, z, p)\in C^2(B_2\times \RR\times \RR^n)$ satisfying 
    \begin{align}
    \begin{split}
      \partial_p A^{ij}|_{B_2\times [-1, 1]\times B_1} \cdot \xi_i\xi_j & \geq  \lambda|\xi|^2\ \ \ \ \forall \xi\in \RR^n \\
      \|A\|_{C^3(\Omega_2)} + \|B\|_{C^2(\Omega_2)} & \leq \Lambda
    \end{split} \label{Appendix_Ell coeff}
    \end{align}
    Let $\scM u := -div(A(x, u, \nabla u)) + B(x, u, \nabla u)$. Consider the variational inequality for a function $u\in C^1(B_2, \RR_+)$ 
    \begin{align}
    \begin{split}
     \int_{B_2} A(x, u, \nabla u)\cdot \nabla \phi + B(x, u, \nabla u)\phi\ dx & \geq 0\ \ \ \forall \phi\in C_c^1(B_2, \RR_+) \\
     \int_{B_2} A(x, u, \nabla u)\cdot \nabla \phi + B(x, u, \nabla u)\phi\ dx & = 0\ \ \ \forall \phi\in C_c^1(\{u>0\})
    \end{split} \label{Appendix_Var ineq of 2nd ell}
    \end{align}
    Clearly, (\ref{Appendix_Constr stny}) implies (\ref{Appendix_Var ineq of 2nd ell}), where $A = \partial_p F$ and $B = \partial_z F$.
    
    \begin{Thm} \label{Thm_Appendix_Reg obst for 2nd ell equ}
     $\exists \delta_0 = \delta_0(\lambda, \Lambda, n, \alpha)\in (0, 1)$ such that if $A, B$ satisfies (\ref{Appendix_Ell coeff}), $0\leq u\in C^{1,\alpha}_{loc}(B_2\setminus \{0\})\cap C^1(B_2)$ satisfies (\ref{Appendix_Var ineq of 2nd ell}) and that $\|u\|_{C^1(B_2)}\leq \delta_0$, then $u\in C^{1,1}_{loc}(B_2)$ and \[
      \|u\|_{C^{1,1}(B_1)} \leq C(\lambda, \Lambda, n, \alpha)     \]
    \end{Thm}
    In view of the proof of lemma \ref{Lem_Strc thm_tilt est}, theorem \ref{Thm_Survey Obst_Reg for non-para} follows from theorem \ref{Thm_Appendix_Reg obst for 2nd ell equ}.
    
    The major effort is made to prove the following lemma,
    \begin{Lem} \label{Lem_Appendix_Reg obst 2nd ell}
     There exists $\delta_0 = \delta_0(\lambda, \Lambda, n, \alpha)\in (0, 1)$, $\eta_0 = \eta_0(\lambda, \Lambda, n, \alpha)\in (0, \delta_0)$ s.t. 
     if $u$ be in theorem \ref{Thm_Appendix_Reg obst for 2nd ell equ} and further satisfying 
     \begin{align}
      \|u\|_{C^{1, \alpha}(B_2)}  \leq 2\delta_0 \label{Appendix_Small C^1,a sol}
     \end{align}
     and $\eta<\eta_0$ such that the elliptic coefficients satisfies
     \begin{align}
      \|\partial_x A\|_{C^2(\Omega_2)} +\|\partial_z A\|_{C^2(\Omega_2)} + \|B\|_{C^2(\Omega_2)}  \leq \eta \label{Appendix_Small low order coeff}
     \end{align}
     Then $\|u\|_{C^{1,1}(B_1)} \leq C(\lambda, \Lambda, n, \alpha)(\|u\|_{C^{1,\alpha}(B_2)} + \eta) $.
    \end{Lem}
    \begin{proof}[Proof of theorem \ref{Thm_Appendix_Reg obst for 2nd ell equ}]
     Observe that if $x_0\in B_1$, $r\in (0, 1)$, let $\tilde{u}(y):= u(x_0 + ry)/r$, then \[
      \scM u (x) = \frac{1}{r} [-div \tilde{A}(y, \tilde{u}(y), \nabla \tilde{u}(y)) + \tilde{B}(y, \tilde{u}(y), \nabla \tilde{u}(y))]_{y=(x-x_0)/r}     \]
     where 
     \begin{align*}
      \tilde{A}(y, z, p)& := A(x_0+ry, rz, p) \\
      \tilde{B}(y, z, p)& := rB(x_0+ry, rz, p)
     \end{align*}
     Hence, by lemma \ref{Lem_Appendix_Reg obst 2nd ell}, for $u$ satisfying the assumptions in theorem \ref{Thm_Appendix_Reg obst for 2nd ell equ}, $u\in C^{1,1}_{loc}(B_2\setminus \{0\})$. Thus, for a.e. $x\in B_2\setminus \{0\}$, $\scM u (x) = B(x, 0, 0)$. Together with (\ref{Appendix_Var ineq of 2nd ell}) and that $\|u\|_{C^1(B_2)} \leq 1$, by a cutting off argument, $\scM u (x) = B(x, 0, 0)\cdot\chi_{\{u=0\}}$ in the weak sense on $B_2$. Hence by standard elliptic estimate \cite[theorem 13.1]{GilbargTrudinger01}, $\exists \beta=\beta(\lambda, \Lambda, n)\in (0 ,1)$ such that $\|u\|_{C^{1, \beta}(B_{3/2})}\leq C(\lambda, \Lambda, n)$. Then repeat the rescaling argument above, by lemma \ref{Lem_Appendix_Reg obst 2nd ell}, theorem \ref{Thm_Appendix_Reg obst for 2nd ell equ} is proved.
    \end{proof}
    
    Now we start to prove lemma \ref{Lem_Appendix_Reg obst 2nd ell}. 
    \begin{Lem} \label{Lem_Appendix_Exist uniq with low bdy reg}
     $\exists \eta_1 = \eta_1(\lambda, \Lambda, n, \alpha)\in (0,1)$ such that if $\eta\leq \eta_1$ and $A, B$ satisfies (\ref{Appendix_Ell coeff}) and (\ref{Appendix_Small low order coeff}), $\|\phi\|_{C^{1, \alpha}(\partial B_1)}\leq \eta_1$. Then the equation
     \begin{align}
     \begin{cases}
      \scM u = 0\ & \text{ in }B_1 \\
      u = \phi\ & \text{ on }\partial B_1
     \end{cases} \label{Appendix_Ell equ with low bdy reg}
     \end{align}
     has a unique solution $u\in C^{1,\alpha}(Clos(B_1))$ satisfying 
     \begin{align}
      \|u\|_{C^{1, \alpha}}\leq C(\lambda, \Lambda, n, \alpha)(\|\phi\|_{C^{1, \alpha}}+ \eta) \label{Appendix_C^1,alpha est} 
     \end{align}
    \end{Lem}
     The uniqueness follows directly from comparison principle, see \cite[Chapter 10]{GilbargTrudinger01};
     We assert that the existence result does not follow directly from the classical Schauder estimate, since the latter requires boundary value to be $C^{2,\alpha}$. Instead, one need to work in the space $H_a^{(b)}$ introduced \cite{GilbargHormander80}, where $a = 2+\alpha$, $b = 1+\alpha$. By \cite[theorem 5.1]{GilbargHormander80} and a fixed point argument, there exists $u\in H^{(b)}_a(B_1)$ satisfying (\ref{Appendix_Ell equ with low bdy reg}) with uniformly bounded $H^{(b)}_a$-norm; and (\ref{Appendix_C^1,alpha est}) follows from \cite[lemma 2.1]{GilbargHormander80}.
     
     \begin{proof}[Proof of lemma \ref{Lem_Appendix_Reg obst 2nd ell}]
      Let $\delta_0, \eta_0\in (0, \eta_1)$ TBD; $u$ be in lemma \ref{Lem_Appendix_Reg obst 2nd ell}.\\
      
\noindent      \textbf{Step 1}
      $\forall x_0\in B_{3/2}\cap\{u=0\}$, $\forall r\in (0, 1/2)$, let $v=v_{x_0, r}\in C^{1,\alpha}(B_r(x_0))$ be the solution of 
      \begin{align*}
      \begin{cases}
       \scM v = 0\ & \text{ in }B_r(x_0) \\
       v = u\ & \text{ on }\partial B_r(x_0)
      \end{cases}
      \end{align*}
      given by lemma \ref{Lem_Appendix_Exist uniq with low bdy reg}. Also, let $\bar{v}$ be the solution of 
      \begin{align*}
      \begin{cases}
       \scM \bar{v} = 0\ & \text{ in }B_r(x_0) \\
       \bar{v} = 0\ & \text{ on }\partial B_r(x_0)       
      \end{cases}
      \end{align*}
      Also note that $\|u\|_{C^1(B_r)}, \|v\|_{C^1(B_r)}, \|\bar{v}\|_{C^1(B_r)}\leq C(\lambda, \Lambda, n, \alpha)(\eta_0 + \delta_0)$ and $\scM u \geq 0$ in the weak sense. Hence, by taking $\eta_0, \delta_0 << 1$ and using the maximum principle \cite[theorem 10.10]{GilbargTrudinger01}, 
      \begin{align}
       0\leq u(x)- v(x) \leq C(\lambda, \Lambda, n, \alpha)\sup_{\partial \{u>0\}} u- v \leq C(\lambda, \Lambda, n, \alpha)\|v^-\|_{C^0(B_r(x_0))} \label{Appendix_Cprsn u-v}
      \end{align}
      for every $x\in B_r(x_0)$; And $\bar{v}\leq v$ on $B_r(x_0)$. Applying the classical $C^0$ estimate on $\tilde{v}(y):=\bar{v}(x_0+ry)/r$, one get
      \begin{align}
       \|\bar{v}\|_{C^0(B_r(x_0))} \leq C(\lambda, \Lambda, n, \alpha)r^2\|B(\cdot, 0, 0)\|_{C^0(B_r)} \label{Appendix_C^0 est bar{v}}
      \end{align}
      Since $v-\bar{v}\geq 0$ on $B_r(x_0)$, $\scM v = \scM \bar{v} = 0$ and $v(x_0)\leq u(x_0)=0$, by (\ref{Appendix_Ell coeff}) and Harnack inequality, 
      \begin{align}
       \|v-\bar{v}\|_{C^0(B_{r/2}(x_0))} \leq C(\lambda, \Lambda, n, \alpha)(v(x_0) - \bar{v}(x_0)) \leq C(\lambda, \Lambda, n, \alpha)|\bar{v}(x_0)| \label{Appendix_Harnack ineq v-bar{v}}
      \end{align}
      Combining (\ref{Appendix_Cprsn u-v}), (\ref{Appendix_C^0 est bar{v}}) and (\ref{Appendix_Harnack ineq v-bar{v}}) one derive
      \begin{align}
       0\leq u \leq C(\lambda, \Lambda, n, \alpha)\|B(\cdot, 0, 0)\|_{C^0(B_r(x_0))}r^2\ \ \ \text{ on }B_{r/2}(x_0)  \label{Appendix_C^0 est on u}
      \end{align}
      
\noindent   \textbf{Step 2}
      $\forall x_0\in \{u=0\}$, $\nabla u(x_0) = 0$; $\forall x_0\in \{u>0\}\cap B_{3/2}$, if $l:=dist(x_0, \{u=0\})<1/8$, let $y_0\in \{u=0\}$ such that $|x_0 - y_0| = l$. Then by (\ref{Appendix_C^0 est on u}), \[
       \|u\|_{C^0(B_l(x_0))}\leq \|u\|_{C^0(B_{2l}(y_0))}\leq C(\lambda, \Lambda, n, \alpha)l^2     \] 
      Hence by interior gradient estimate, 
      \begin{align}
       \frac{1}{l}\|\nabla u\|_{C^0(B_{l/2}(x_0))} + \|\nabla^2 u\|_{C^0(B_{l/2}(x_0))} \leq C(\lambda, \Lambda, n, \alpha) \label{Appendix_grad est on u}
      \end{align}
      Now that, $\forall x_1, x_2\in B_1$, $\rho:=|x_1-x_2|<1/20$, $l_i:=dist(x_i, \{u=0\})$, $l_1\leq l_2$. 
      \begin{enumerate}
      \item[(i)] If $l_2 \geq 1/8$, then directly by gradient estimate \[
        |\nabla u(x_1)- \nabla u(x_2)|\leq \rho \sup_{B_{\rho}(x_2)}\|\nabla^2 u\| \leq C(\lambda, \Lambda, n, \alpha)|x_1-x_2|    \]
      \item[(2)] If $2\rho\leq l_2< 1/8$, then by (\ref{Appendix_grad est on u}), \[
        |\nabla u(x_1)- \nabla u(x_2)|\leq \rho \sup_{B_{l_2/2}(x_2)}\|\nabla^2 u\| \leq C(\lambda, \Lambda, n, \alpha)|x_1-x_2|    \]
      \item[(3)] If $l_2\leq 2\rho$, then also by (\ref{Appendix_grad est on u}), \[
        |\nabla u(x_1)-\nabla u(x_2)| \leq |\nabla u(x_1)|+|\nabla u(x_2)|\leq C(\lambda, \Lambda, n, \alpha)(l_1+l_2)\leq C(\lambda, \Lambda, n, \alpha)|x_1-x_2|    \]
      \end{enumerate}
      This proves the $C^{1, 1}$ bound of $u$.
     \end{proof}
     
     \begin{proof}[Proof of theorem \ref{Thm_Survey Obst_Reg for para}]
      The regularity results in literature focus on minimizing boundaries; for general closed integral current, the strategy is to first decompose them into sum of boundary of Caccioppoli sets. One subtlety is that a priori the graphs may have incompatible orientations. We rule out this by choosing $U_y$ small enough.
      
      Recall that $M^{n+1}\subset \RR^L$ is a compact submanifold with boundary, extended to a closed manifold $\tilde{M}^{n+1}\subset \RR^L$; $\tilde{U}\subset \tilde{M}$ is relatively open and $U:=\tilde{U}\cap M$. Suppose $U\cap \partial M \neq \emptyset$. $T\in \cZ_n(U)$ minimize mass among $\{S\in \cZ_n(U): spt(T-S)\subset \subset U\}$. Call such $T$ \textbf{constrained minimizing} in $U$.
      
      WLOG $U, \tilde{U}, V:=\tilde{U}\setminus U$ are contractible and $\bfM_U(T)>0$. By decomposition theorem \cite[Chap 6, 3.14]{Simon83_GMT}, there's a decreasing sequence of $\scH^{n+1}$-measurable subsets $\{U_j\subset \tilde{U}\}_{j\in \ZZ}$ such that \[
       T = \sum_{j\in\ZZ} \partial [U_j];\ \ \ \ \|T\|= \sum_{j\in \ZZ} \|\partial [U_j]\|     \]
      In particular, $\bfM_W(T) = \sum_{j\in \ZZ}\bfM_W(\partial [U_j])$ for all $W\subset \subset \tilde{U}$, $spt(\partial [U_j])\subset U$ and $T_j := \partial [U_j]$ are also constrained minimizers in $U$ (along $\partial M$). Thus for each $j$, either $spt[U_j]$ or $spt([\tilde{U}]- U_j])$ contains $V$. WLOG, $spt(\partial [U_0])\cap \partial M \neq \emptyset$ and $spt[U_0]\supset V$. Hence $\forall j\leq 0$, $spt[U_j]\supset V$; and $\forall j>0$, either $spt[U_j]\supset V$ or $spt[U_j]\subset spt[U_0]\setminus V$.\\
      
\noindent   \textbf{Claim:} $\forall \cK\subset U\cap \partial M$ compact, $\exists\ r_1 = r_1(M, U, \cK)>0$ such that if $T_0 = \partial [U_0]$, $T_1 = \partial [U_1]$ are constrained minimizers in $U$ and $U_0\subset V$, $U_1\subset U_0\setminus V$; $y\in \cK\cap spt(T_0)$. Then $U_1\cap B_{r_1}(y)=\emptyset$. 
      
      Clearly, with this claim, for each $j\in \ZZ$, $T_j\llcorner B_{r_1}(y) = \partial [U_j']\llcorner B_{r_1}(y)$ is constrained minimizing in $B_{r_1}(y)$, where $U_j' := U_j\cap B_{r_1}(y)$ is either empty or containing $V$. Hence, by \cite{Miranda71} and \cite{Tamanini82}, $T_j\llcorner B_{r_1}(y)$ is $C^{1,\alpha}$ for some $\alpha\in (0, 1)$. Then by \ref{Thm_Survey Obst_Reg for non-para}, $T$ is $C^{1, 1}$ multiple graphs restricting to a smaller subset $U_y$. This proves (1) of theorem \ref{Thm_Survey Obst_Reg for para}; and (2) follows from unique continuation of minimal hypersurfaces.\\
      
\noindent      \textbf{Proof of the claim:} By \cite{Miranda71}, for every $\delta\in (0,1)$, $\exists\ r_2 = r_2(M, U, \cK, \delta)>0$ and $u_0\in C^1(U\cap \partial M)$ such that 
      \begin{align}
       T_0\llcorner B_{r_2}(y) = [graph_{\partial M}(u)]\llcorner B_{r_2}(y) \ \text{ with }\  |\nabla u|\leq \delta  \label{Appendix_Miranda '71, small grad}
      \end{align}
      Assuming the claim fails, then $\exists\ y_j\in \cK$, $T_i^{(j)} = \partial [U_i^{(j)}]$ being constrained minimizing boundary in $U$, where $i=0,1$, $j\geq 1$, $U_0^{(j)}\subset V$, $U_1^{(j)}\subset U_0^{(j)}\setminus V$ and $U_1^{(j)}\cap B_{1/j}(y_j)\neq \emptyset$. Then $spt(T_1^{(j)})\cap B_{1/j}(y_j)\neq \emptyset$.

      Let $j\to \infty$, $y_j\to y_{\infty}\in \cK$, and consider $(\eta_{y_j, 1/j})_{\sharp}T_1^{(j)}\to T_1^{(\infty)}$ being constrained minimizing in $T_{y_{\infty}}^+M \cong \RR^{n+1}_+$. The convergent is both in current and in varifold sense. Therefore, by (\ref{Appendix_Miranda '71, small grad}) and that $U_1^{(j)}\subset U_0^{(j)}\setminus V$, $T_1^{(\infty)}=\partial [U_1^{(\infty)}]$ for some $spt(U_1^{(\infty)})\subset T_{y_\infty}\partial M$, which is $n+1$ negligible. In particular, $T_1^{(\infty)} = 0$; On the other hand, by $spt(T_1^{(j)})\cap B_{1/j}(y_j)\neq \emptyset$, lemma \ref{Lem_Constr stny_Bdy Mon} and corollary \ref{Cor_Constr stny_Bdy density lbd}, $T_1^{(\infty)}\neq 0$ and we get a contradiction.
     \end{proof}

  \bibliographystyle{plain}
  \bibliography{GMT_Obst}

\begin{thebibliography}{10}

\bibitem{Allard72}
William~K. Allard.
\newblock On the first variation of a varifold.
\newblock {\em Annals of Mathematics}, 95(3):417--491, 1972.

\bibitem{Almgren75}
F.~J. Almgren.
\newblock Existence and regularity almost everywhere of solutions to elliptic
  variational problems with constraints.
\newblock {\em Bull. Amer. Math. Soc.}, 81(1):151--154, 01 1975.

\bibitem{Almgren62}
Frederick~J. Almgren.
\newblock The homotopy groups of the integral cycle groups.
\newblock {\em Topology}, 1(4):257 -- 299, 1962.

\bibitem{Almgren65_Varifolds}
Frederick~Justin Almgren, Jr.
\newblock The theory of varifolds, 1965.
\newblock mimeographed notes, Princeton.

\bibitem{BarozziMassari82}
Elisabetta Barozzi and U.~Massari.
\newblock Regularity of minimal boundaries with obstacles.
\newblock {\em Rendiconti del Seminario Matematico della Universit\`a di
  Padova}, 66:129--135, 1982.

\bibitem{BernsteinWang20}
Jacob Bernstein and Lu~Wang.
\newblock A mountain-pass theorem for asymptotically conical self-expanders.
\newblock {\em arXiv preprint arXiv:2003.13857}, 2020.

\bibitem{Bombieri82}
Enrico Bombieri.
\newblock Regularity theory for almost minimal currents.
\newblock {\em Archive for Rational Mechanics and Analysis}, 78(2):99--130,
  1982.

\bibitem{BrezisKinderl74}
Ha{\"\i}m Br{\'e}zis, David Kinderlehrer, and Hans Lewy.
\newblock The smoothness of solutions to nonlinear variational inequalities.
\newblock {\em Indiana university mathematics journal}, 23(9):831--844, 1974.

\bibitem{CaffarKinderle80}
LA~Caffareli and David Kinderlehrer.
\newblock Potential methods in variational inequalities.
\newblock {\em Journal d’Analyse Math{\'e}matique}, 37(1):285--295, 1980.

\bibitem{DeLellisRamic18}
Camillo De~Lellis and Jusuf Ramic.
\newblock Min-max theory for minimal hypersurfaces with boundary.
\newblock In {\em Annales de l'Institut Fourier}, volume~68, pages 1909--1986,
  2018.

\bibitem{Federer69}
Herbert Federer.
\newblock {\em Geometric Measure Theory}, volume 153 of {\em Grundlehren der
  Math. Wiss.}
\newblock Springer-Verlag, New York, 1969.

\bibitem{FedererFleming60}
Herbert Federer and Wendell~H. Fleming.
\newblock Normal and integral currents.
\newblock {\em Ann. of Math. (2)}, 72:458--520, 1960.

\bibitem{GiaquintaPepe71}
Mariano Giaquinta and L.~Pepe.
\newblock Esistenza e regolarit\`a per il problema dell'area minima con
  ostacoli in $n$ variabili.
\newblock {\em Annali della Scuola Normale Superiore di Pisa - Classe di
  Scienze}, Ser. 3, 25(3):481--507, 1971.

\bibitem{GilbargHormander80}
David Gilbarg and Lars H{\"o}rmander.
\newblock Intermediate schauder estimates.
\newblock {\em Archive for Rational Mechanics and Analysis}, 74(4):297 -- 318,
  1980.

\bibitem{GilbargTrudinger01}
David Gilbarg and Neil~S. Trudinger.
\newblock {\em {{E}}lliptic {{P}}artial {{D}}ifferential {{E}}quations of
  {{S}}econd {{O}}rder}.
\newblock Springer-Verlag, New York, 2001.
\newblock reprint of the 1998 edition.

\bibitem{Giusti72}
E~Giusti.
\newblock Minimal surfaces with obstacles.
\newblock In {\em Geometric Measure Theory and Minimal Surfaces}, pages
  119--153. Springer, 2010.

\bibitem{GuangLiZhou18Crelle}
Qiang Guang, Martin~Man chun Li, and Xin Zhou.
\newblock Curvature estimates for stable free boundary minimal hypersurfaces.
\newblock {\em Journal für die reine und angewandte Mathematik}, 2020(759):245
  -- 264, 2020.

\bibitem{Ilmanen96}
Tom Ilmanen.
\newblock A strong maximum principle for singular minimal hypersurfaces.
\newblock {\em Calculus of Variations and Partial Differential Equations},
  4(5):443--467, 1996.

\bibitem{IrieMarquesNeves18}
K.~Irie, F.~C. Marques, and A.~Neves.
\newblock Density of minimal hypersurfaces for generic metrics.
\newblock {\em Ann. of Math. (2)}, 187:963--972, 2018.

\bibitem{LewyStampacchia71}
H.~Lewy and G.~Stampacchia.
\newblock On existence and smoothness of solutions of some non-coercive
  variational inequalities.
\newblock {\em Arch. Rational Mech. Anal.}, (41):241--253, 1971.

\bibitem{ZhouLi16_FreeBdyMinmax}
Martin Li and Xin Zhou.
\newblock Min-max theory for free boundary minimal hypersurfaces i-regularity
  theory.
\newblock {\em arXiv preprint arXiv:1611.02612}, 2016.

\bibitem{LiYy19_Exis}
Yangyang Li.
\newblock Existence of infinitely many minimal hypersurfaces in
  higher-dimensional closed manifolds with generic metrics.
\newblock {\em arXiv preprint arXiv:1901.08440}, 2019.

\bibitem{Lin85}
FANG-HUA LIN.
\newblock {\em REGULARITY FOR A CLASS OF PARAMETRIC OBSTACLE PROBLEMS
  (INTEGRAND, INTEGRAL CURRENT, PRESCRIBED MEAN CURVATURE, MINIMAL SURFACE
  SYSTEM)}.
\newblock PhD thesis, 1985.

\bibitem{MarquesNevesSong19}
Fernando~C Marques, Andr{\'e} Neves, and Antoine Song.
\newblock Equidistribution of minimal hypersurfaces for generic metrics.
\newblock {\em Inventiones mathematicae}, 216(2):421--443, 2019.

\bibitem{MarquesNeves14}
Fernando~C. Marques and André Neves.
\newblock Min-max theory and the {W}illmore conjecture.
\newblock {\em Ann. of Math. (2)}, 179(2):683--782, 2014.

\bibitem{MarquesNeves17_Infinite}
Fernando~C. Marques and André Neves.
\newblock Existence of infinitely many minimal hypersurfaces in positive
  {R}icci curvature.
\newblock {\em Invent. Math.}, 209(2):577--616, 2017.

\bibitem{Miranda71}
M.~Miranda.
\newblock Frontiere minimali con ostacoli.
\newblock {\em Ann. Univ. Ferrara}, (16):29--37, 1971.

\bibitem{Montezuma18}
Rafael Montezuma.
\newblock A mountain pass theorem for minimal hypersurfaces with fixed
  boundary.
\newblock {\em Calculus of Variations and Partial Differential Equations},
  59(6):1--30, 2020.

\bibitem{Pitts81}
Jon~T. Pitts.
\newblock {\em Existence and Regularity of Minimal Surfaces on {R}iemannian
  Manifolds}, volume~27 of {\em Math. Notes}.
\newblock Princeton Univ. Press, Princeton, N.J., 1981.

\bibitem{SchoenSimon81}
Richard Schoen and Leon Simon.
\newblock Regularity of stable minimal hypersurfaces.
\newblock {\em Comm. Pure Appl. Math.}, 34(6):741--797, 1981.

\bibitem{Simon83_GMT}
Leon Simon.
\newblock {\em Lectures on Geometric Measure Theory}, volume~3 of {\em Proc.
  Centre for Mathematical Analysis, Australian National University}.
\newblock Australian National University, Centre for Mathematical Analysis,
  Canberra, 1983.

\bibitem{Song18_YauConj}
Antoine Song.
\newblock Existence of infinitely many minimal hypersurfaces in closed
  manifolds.
\newblock {\em arXiv preprint arXiv:1806.08816}, 2018.

\bibitem{Tamanini82}
Italo Tamanini.
\newblock Boundaries of caccioppoli sets with hölder-continuois normal vector.
\newblock {\em Journal für die reine und angewandte Mathematik}, 334:27--39,
  1982.

\bibitem{White94}
Brian White.
\newblock A strong minimax property of nondegenerate minimal submanifolds.
\newblock {\em Journal für die reine und angewandte Mathematik}, 457:203--218,
  1994.

\bibitem{White09}
Brian White.
\newblock The maximum principle for minimal varieties of arbitrary codimension.
\newblock {\em Communications in Analysis and Geometry}, 18(3):421--432, 2010.

\bibitem{Zhou19_multi1}
Xin Zhou.
\newblock On the multiplicity one conjecture in min-max theory.
\newblock {\em arXiv preprint arXiv:1901.01173}, 2019.

\bibitem{ZhouZhu18}
Xin Zhou and Jonathan~J. Zhu.
\newblock Existence of hypersurfaces with prescribed mean curvature i –
  generic min-max.
\newblock {\em Cambridge Journal of Mathematics}, 8(2):311 -- 362, 2020.

\end{thebibliography}
\end{document}